\documentclass[11pt]{article}

\title{\vspace{-50pt}On finite generation of the Johnson filtrations}
\author{Thomas Church,\ Mikhail Ershov,\ and Andrew Putman\thanks{TC is supported in part by NSF grant DMS-1350138, the Alfred P.\ Sloan Foundation, the Institute for Advanced Study, and the Friends of the Institute. AP is supported in part by NSF grant DMS-1737434.}\vspace{-6pt}}
\date{January 11, 2021}

\usepackage{amsmath,amssymb,amsthm,amscd,amsfonts}
\usepackage{epsfig,pinlabel}
\usepackage[hmargin=1.25in, vmargin=1in]{geometry}
\usepackage[font=small,format=plain,labelfont=bf,up,textfont=it,up]{caption}
\usepackage{type1cm}
\usepackage{calc}
\usepackage{enumitem}
\usepackage{bm}
\usepackage{url}
\usepackage[all,cmtip]{xy}
\usepackage{tabularx}
\usepackage[usenames,dvipsnames]{xcolor}
\numberwithin{equation}{section}
\usepackage{inputenc}

\usepackage[
colorlinks=true, linkcolor=black, citecolor=black, urlcolor=black]{hyperref}

\newlist{compactitem}{itemize}{3}
\setlist[compactitem]{nosep}
\setlist[compactitem,1]{label=\textbullet}
\setlist[compactitem,2]{label=--}
\setlist[compactitem,3]{label=\ensuremath{\ast}}

\newlist{compactdesc}{description}{3}
\setlist[compactdesc]{nosep}

\newlist{compactenum}{enumerate}{3}
\setlist[compactenum]{nosep}
\setlist[compactenum,1]{label=\arabic*.}
\setlist[compactenum,2]{label=(\alph*)}
\setlist[compactenum,3]{label=\roman*.}

\usepackage{bookmark}
\bookmarksetup{depth=2}

\newcommand{\mylabel}[2]{#2\def\@currentlabel{#2}\label{#1}}

\theoremstyle{plain}
\newtheorem{theorem}{Theorem}[section]
\newtheorem{theoremrepeat}{Theorem}
\newtheorem{maintheorem}{Theorem}
\newtheorem{proposition}[theorem]{Proposition}
\newtheorem{lemma}[theorem]{Lemma}
\newtheorem{corollary}[theorem]{Corollary}

\newtheorem{innerthm}{Theorem}
\newenvironment{theorem-prime}[1]{\renewcommand\theinnerthm{\ref*{#1}$'$}\innerthm}{\endinnerthm}

\theoremstyle{definition}

\newtheorem{definition}[theorem]{Definition}

\theoremstyle{remark}
\newtheorem{remark}[theorem]{Remark}
\newtheorem{remarks}[theorem]{Remarks}

\DeclareMathOperator{\Hom}{Hom}

\DeclareMathOperator{\Ker}{ker}

\DeclareMathOperator{\Mod}{Mod}

\newcommand\Torelli{\ensuremath{{\mathcal I}}}
\newcommand\JohnsonKer{\ensuremath{{\mathcal K}}}
\newcommand\Johnson{\ensuremath{J}}
\DeclareMathOperator{\IA}{IA}
\DeclareMathOperator{\JIA}{JIA}
\DeclareMathOperator{\IO}{IO}
\DeclareMathOperator{\Sp}{Sp}
\DeclareMathOperator{\GL}{GL}
\DeclareMathOperator{\SL}{SL}

\DeclareMathOperator{\comp}{comp}

\newcommand\R{\ensuremath{\mathbb{R}}}
\newcommand\C{\ensuremath{\mathbb{C}}}
\newcommand\Z{\ensuremath{\mathbb{Z}}}
\newcommand\Q{\ensuremath{\mathbb{Q}}}
\newcommand\N{\ensuremath{\mathbb{N}}}

\DeclareMathOperator{\HH}{H}

\DeclareMathOperator{\Aut}{Aut}
\DeclareMathOperator{\SAut}{SAut}
\DeclareMathOperator{\Lie}{Lie}
\DeclareMathOperator{\Out}{Out}
\DeclareMathOperator{\Interior}{Int}

\newcommand\Set[2]{\ensuremath{\{\text{#1 $|$ #2}\}}}
\newcommand\lSet[2]{\ensuremath{\langle\text{#1 $|$ #2}\rangle}}

\newcommand\Figure[4]{
\begin{figure}[t]
\centering
\centerline{\psfig{file=#2,scale=#4}}
\caption{#3}
\label{#1}
\end{figure}}

\newcommand{\cL}{\ensuremath{\mathcal{L}}}

\newcommand{\cZ}{\ensuremath{\mathcal{Z}}}
\newcommand{\cLJ}{\mathcal{JL}}

\newcommand{\para}[1]{\medskip\noindent\textbf{#1.}}
\newcommand{\into}{\hookrightarrow}
\newcommand{\onto}{\twoheadrightarrow}
\newcommand{\bwedge}{\textstyle{\bigwedge}}
\newcommand{\abs}[1]{\left\lvert#1\right\rvert}
\newcommand{\normal}{\lhd}

\DeclareMathOperator{\ab}{ab}
\newcommand{\coloneq}{\mathrel{\mathop:}\mkern-1.2mu=}
\newcommand{\AutFn}{\ensuremath{\Aut(F_n)}}
\newcommand{\SAutFn}{\ensuremath{\SAut(F_n)}}
\newcommand{\iso}{\cong}
\newcommand{\subgp}{\mathcal{G}}

\newcommand\cC{\ensuremath{\mathcal{C}}}

\newcommand{\orho}{\ensuremath{\overline{\rho}}}

\newcommand{\BP}[1]{BP(\Torelli_{#1}^1)}
\DeclareMathOperator{\Conj}{Conj}

\newcommand{\arXiv}[1]{\href{http://arxiv.org/abs/#1}{arXiv:#1}}
\newcommand{\myemail}[1]{\href{mailto:#1}{\nolinkurl{#1}}}

\begin{document}

\maketitle

\vspace{-25pt}
\begin{abstract}
We prove that every term of the lower central series and Johnson filtrations of the Torelli subgroups of the mapping
class group and the automorphism group of a free group is finitely generated in a linear stable range.  This was originally
proved for the second terms by Ershov and He.
\end{abstract}

\section{Introduction}

\subsection{The main results}
Let $\Sigma_g^b$ be a compact oriented genus $g$ surface with either $b=0$ or $b=1$ boundary components.
We will often omit $b$ from our notation when $b=0$.
The {\em mapping class group}
of $\Sigma_g^b$, denoted $\Mod_g^b$,
is the group of isotopy classes of orientation-preserving diffeomorphisms of $\Sigma_g^b$ that
fix $\partial \Sigma_g^b$ pointwise.  The group $\Mod_g^b$ acts on $\HH_1(\Sigma_g^b;\Z)$ and preserves
the algebraic intersection form.  Since $b \leq 1$, the algebraic intersection form is a 
symplectic form, and thus this action induces a homomorphism $\Mod_g^b \rightarrow \Sp_{2g}(\Z)$ which is classically known
to be surjective. The kernel of this homomorphism is the Torelli group $\Torelli_g^b$.  We therefore have a short exact sequence
\[1 \longrightarrow \Torelli_g^b \longrightarrow \Mod_g^b \longrightarrow \Sp_{2g}(\Z) \longrightarrow 1.\]
See~\cite{FarbMargalitPrimer} for a survey of the mapping class group and Torelli group.

\para{Lower central series}
For any group $G$, the {\em lower central series} of $G$ is the sequence
\[G = \gamma_1 G \supseteq \gamma_2 G \supseteq \gamma_3 G \supseteq \cdots\]
of subgroups of $G$ defined via the inductive formula
\[\gamma_1 G = G \quad \text{and} \quad \gamma_{k+1}G = [\gamma_k G,G] \quad \quad (k \geq 1).\]
Equivalently, $\gamma_{k+1}G$ is the smallest normal subgroup of $G$ such that $G/\gamma_{k+1}G$ is
$k$-step nilpotent.  The lower central series of $\Torelli_g^b$ has connections to number theory
(see, e.g.,~\cite{Matsumoto}), to 3-manifolds (see, e.g.,~\cite{GaroufalidisLevine}), and to 
the Hodge theory of the moduli space of curves (see, e.g.,~\cite{HainCompletions,HainInfinitesimal}). Despite these connections,
the structure of the lower central series of $\Torelli_g^b$ remains mysterious.  One
of the few structural results known about it is a theorem of Hain~\cite{HainInfinitesimal} giving
a finite presentation for its associated Malcev Lie algebra.

\para{Finite generation}
A classical theorem of Dehn~\cite{DehnGen} from 1938 says that $\Mod_g^b$ is finitely generated.  Since
$\Torelli_g^b$ is an infinite-index normal subgroup of $\Mod_g^b$, there is no reason to expect $\Torelli_g^b$ to be
finitely generated, and indeed McCullough and Miller~\cite{McCulloughMiller}
proved that $\Torelli_2^b$ is not finitely generated.  However, a deep and surprising theorem of
Johnson~\cite{JohnsonFinite} says that $\Torelli_g^b$ is finitely generated for $g \geq 3$. 

\para{Johnson kernel}
The {\em Johnson kernel}, denoted $\JohnsonKer_g^b$, is the subgroup of $\Mod_g^b$ generated
by Dehn twists about simple closed separating curves.
Whether or not $\JohnsonKer_g^b$ is finitely generated for $g \geq 3$ is a well-known question. The case $b=0$ was first mentioned by McCullough and Miller in 1986 \cite{McCulloughMiller}, and appeared in Morita's 1994 problem list \cite[Question 10]{MoritaProblems} and 1998 survey \cite[Problem 2.2(i)]{MoritaSurveyProspect}.
Johnson \cite{JohnsonAbel} proved that $[\Torelli_g^b,\Torelli_g^b]$ is a finite-index subgroup
of $\JohnsonKer_g^b$.  It follows that $\JohnsonKer_g^b$ is finitely generated
if and only if $[\Torelli_g^b,\Torelli_g^b]$ is finitely generated.

Initially, various people conjectured that the answer to this finite generation question is negative for all $g\geq 3$.  This expectation shifted towards a positive answer after the deep work of Dimca and Papadima~\cite{DimcaPapadimaKg},
who proved that $\HH_1(\JohnsonKer_g;\Q)$ is finite dimensional for $g \geq 4$. Dimca, Hain, and Papadima~\cite{DimcaPapadimaHainKg} later gave a description of $\HH_1(\JohnsonKer_g;\Q)$ as a $\Torelli_g/\JohnsonKer_g$-module for $g \geq 6$, which
recently was made more explicit by Morita--Sakasai--Suzuki \cite[Theorem 1.4]{MoritaSakasaiSuzuki}.

Ershov and He~\cite{EH} recently proved that every subgroup of $\Torelli_g^b$ containing $[\Torelli_g^b,\Torelli_g^b]$ 
(in particular, $[\Torelli_g^b,\Torelli_g^b]$ itself and $\JohnsonKer_g^b$) is indeed finitely generated for $g\geq 12$. Our first theorem extends this result to all $g\geq 4$ via a new and simpler proof.  Morita's question is thus now settled with the exception of the single case $g=3$.

\begin{maintheorem}
\label{maintheorem:johnsonkernel}
For $g \geq 4$ and $b \in \{0,1\}$, every subgroup of $\Torelli_g^b$ containing $[\Torelli_g^b,\Torelli_g^b]$ 
is finitely generated.  In particular, $[\Torelli_g^b,\Torelli_g^b]$ and $\JohnsonKer_g^b$ are finitely generated.
\end{maintheorem}

\para{Deeper in the lower central series}
Another result in \cite{EH} asserts that if $k\geq 3$ and $g\geq 8k-4$, then the abelianization of 
any subgroup of $\Torelli_g^b$ containing $\gamma_k\Torelli_g^b$ is finitely generated.
We will prove that any subgroup of $\mathcal I_g^b$ containing $\gamma_k  \mathcal I_g^b$ is actually finitely generated (in fact, with a better range for $g$).

\begin{maintheorem}
\label{maintheorem:torellilcs}
For $k \geq 3$ and $g \geq 2k-1$ and $b \in \{0,1\}$, every subgroup of $\Torelli_g^b$ containing
$\gamma_k\Torelli_g^b$ is finitely generated.  In particular, $\gamma_k \Torelli_g^b$ is finitely generated.
\end{maintheorem}

\begin{remark}
Since every subgroup containing $\gamma_2\Torelli_g^b$ also contains $\gamma_3\Torelli_g^b$,
Theorem~\ref{maintheorem:torellilcs} implies Theorem~\ref{maintheorem:johnsonkernel} except
in the borderline case $g=4$, where the more abstract arguments used to prove
Theorem \ref{maintheorem:torellilcs} do not work.  In addition to handling this
one last case, our proof
of Theorem~\ref{maintheorem:johnsonkernel} is considerably simpler than our proof 
of Theorem~\ref{maintheorem:torellilcs} while using many of the same ideas.  It
thus provides a concise introduction to our general approach.
\end{remark}

\para{The Johnson filtration}
We want to highlight an important special case of Theorem~\ref{maintheorem:torellilcs}.
Fix some $g \geq 0$ and $b \in \{0,1\}$.  Pick a basepoint $\ast \in \Sigma_g^b$; if $b=1$, then
choose $\ast$ such that it lies in $\partial \Sigma_g^b$.  Define $\pi = \pi_1(\Sigma_g^b,\ast)$.  Since
$\Mod_g^1$ is built from diffeomorphisms that fix $\partial \Sigma_g^1$ and thus in particular fix $\ast$, there
is a homomorphism $\Mod_g^1 \rightarrow \Aut(\pi)$.  For closed surfaces, there is no fixed basepoint, so
we only obtain a homomorphism $\Mod_g \rightarrow \Out(\pi)$.  In both cases, this action preserves
the lower central series of $\pi$, so we obtain homomorphisms
\[\psi_g^1[k] \colon \Mod_g^1 \rightarrow \Aut(\pi/\gamma_k\pi) \quad \text{and} \quad \psi_g[k]\colon \Mod_g \rightarrow \Out(\pi/\gamma_k\pi).\]
The $k^{\text{th}}$ term of the {\em Johnson filtration} of $\Mod_g^b$, denoted $\Johnson_g^b(k)$, is the kernel
of $\psi_g^b[k+1]$.  This filtration was introduced in 1981 by Johnson \cite{JohnsonSurvey}.
Chasing the definitions, we find that $\Johnson_g^b(1) = \Torelli_g^b$.  Moreover,
Johnson~\cite{JohnsonKer} proved that $\Johnson_g^b(2) = \JohnsonKer_g^b$.
It is easy to see that $\gamma_k \Torelli_g^b\subseteq\Johnson_g^b(k)$ for all $k$, but these filtrations are
known not to be equal.  In fact, Hain proved that they even define inequivalent topologies
on $\Torelli_g^b$; see \cite[Theorem 14.6 plus \S14.4]{HainInfinitesimal}.  Since $\gamma_k\Torelli_g^b\subseteq\Johnson_g^b(k)$, the following result is a special case of Theorem~\ref{maintheorem:torellilcs}.

\begin{maintheorem}
\label{maintheorem:torellijohnson}
For $k \geq 3$ and $g \geq 2k-1$ and $b \in \{0,1\}$, the group $\Johnson_g^b(k)$ is finitely generated.
\end{maintheorem}

\para{Automorphism groups of free groups}
Let $F_n$ be a free group on $n$ generators.  
The group $\AutFn$ acts on the abelianization $F_n^{\ab} = \Z^n$.  The kernel of this action is the
{\em Torelli subgroup} of $\AutFn$ and is denoted $\IA_n$. A classical theorem of
Magnus~\cite{MagnusGenerators} from 1935 says that $\IA_n$ is finitely generated for all $n$.
Building on the aforementioned work of Dimca and Papadima~\cite{DimcaPapadimaKg} for the mapping class group,
Papadima and Suciu~\cite{PapadimaSuciuKg} proved that $\HH_1([\IA_n,\IA_n];\Q)$ is finite-dimensional
for $n \geq 5$. 

Just like for the mapping class group, Ershov and He~\cite{EH} proved 
that any subgroup of $\IA_n$ containing $[\IA_n,\IA_n]$ is finitely generated for $n \geq 4$.
They also proved that the abelianization of any subgroup of $\IA_n$ containing $\gamma_k \IA_n$ is finitely generated for $n \geq 8k-4$.
The following theorem extends these results from \cite{EH} by both improving the range and strengthening the conclusion; it is a direct counterpart of Theorems~\ref{maintheorem:johnsonkernel}~and~\ref{maintheorem:torellilcs}.

\begin{maintheorem}
\label{maintheorem:ialcs}
For $k \geq 3$ and $n \geq 4k-3$, or for $k=2$ and $n\geq 4$, every subgroup of $\IA_n$ containing $\gamma_k \IA_n$ is finitely generated.  In
particular, $\gamma_k \IA_n$ is finitely generated.
\end{maintheorem}

\begin{remark}
One can also consider the Torelli subgroup $\IO_n$ of $\Out(F_n)$.  The homomorphism $\IA_n \rightarrow \IO_n$
is surjective, so Theorem~\ref{maintheorem:ialcs} also implies a similar result for $\gamma_k \IO_n$.
\end{remark}

\para{Johnson filtration for automorphism group of free group}
Similarly to the mapping class group, there is a natural homomorphism
\[\psi_n[k]\colon \AutFn \rightarrow \Aut(F_n / \gamma_k F_n).\]
The $k^{\text{th}}$ term of the {\em Johnson filtration} for $\AutFn$, denoted
$\JIA_n(k)$, is the kernel of ${\psi_n[k+1]}$.  This filtration was actually introduced
by Andreadakis~\cite{Andreadakis} in 1965, much earlier than the Johnson filtration for
the mapping class group.  It is well known that $\gamma_k \IA_n \subseteq \JIA_n(k)$, and Bachmuth~\cite{Bachmuth} and Andreadakis~\cite{Andreadakis} independently proved that $\gamma_2\IA_n=\JIA_n(2)$.  
Recently Satoh \cite{SatohJIA} proved that $\gamma_3 \IA_n = \JIA_n(3)$, improving an earlier
result of Pettet \cite{Pettet} saying that $\gamma_3 \IA_n$ has finite index in
$\JIA_n(3)$.  However,  
recent computer calculations of Bartholdi \cite{BartholdiErratum} (making key
use of results of Day and Putman \cite{DayPutmanH2IA}) show that these filtrations
are not commensurable for $n=3$. 
It is an open problem whether or not these two filtrations are equal (or at least
commensurable) for $n \geq 4$. Since $\gamma_k \IA_n \subseteq \JIA_n(k)$, Theorem~\ref{maintheorem:ialcs}
in particular applies to all subgroups containing $\JIA_n(k)$.  However, in this special case we
are able to prove finite generation with a better range for $n$.

\begin{maintheorem}
\label{maintheorem:iajohnson}
For $k \geq 2$ and $n \geq 2k+3$, every subgroup of $\IA_n$ containing $\JIA_n(k)$ is finitely generated.  In
particular, $\JIA_n(k)$ is finitely generated.
\end{maintheorem}

\begin{remark}
For $k=2$ and $k=3$ we have $2k+3\geq 4k-3$, so in these cases Theorem~\ref{maintheorem:iajohnson} follows from Theorem~\ref{maintheorem:ialcs}.
\end{remark}

\subsection{Outline of the proof}
\label{section:outlineproof}

We now discuss the ideas behind the proofs of our theorems.

\para{Initial reductions}
We first point out two reductions which show that it suffices to prove Theorem~\ref{maintheorem:torellilcs}
for $b=1$ and for the group $\gamma_k\Torelli_g^1$.  
\begin{compactitem}
\item If $\subgp$ is a group satisfying $\gamma_k\Torelli_g^b \subseteq \subgp \subseteq \Torelli_g^b$, then letting
$\overline{\subgp}$ be the image of $\subgp$ in $\Torelli_g^b / \gamma_k \Torelli_g^b$ we have a short
exact sequence
\[1 \longrightarrow \gamma_k\Torelli_g^b \longrightarrow \subgp \longrightarrow \overline{\subgp} \longrightarrow 1.\]
Since $\Torelli_g^b / \gamma_k\Torelli_g^b$ is a finitely generated nilpotent group, its subgroup
$\overline{\subgp}$ is also finitely generated.  To prove that $\subgp$ is finitely generated, it
is thus enough to prove that $\gamma_k \Torelli_g^b$ is finitely generated.
\item The homomorphism $\Torelli_g^1 \rightarrow \Torelli_g$ obtained by gluing a disc
to $\partial \Sigma_g^1$ is surjective, and thus its restriction $\gamma_k \Torelli_g^1\to \gamma_k\Torelli_g$ is also surjective.  To prove that $\gamma_k \Torelli_g$ is finitely
generated, it is thus enough to prove that $\gamma_k\Torelli_g^1$ is finitely generated.
\end{compactitem}
Similarly, it suffices to prove Theorem~\ref{maintheorem:johnsonkernel} for $[\Torelli_g^1,\Torelli_g^1]$, Theorem~\ref{maintheorem:ialcs} for $\gamma_k \IA_n$, and Theorem~\ref{maintheorem:iajohnson} for $\subgp=\JIA_n(k)$.

\para{$\bm{[n]}$-groups}
All of our main theorems except Theorem~\ref{maintheorem:johnsonkernel} will be deduced from a general result (Corollary~\ref{corollary:heart}
below) which deals with \emph{$[n]$-groups}.  Here $[n]$ denotes the set $\{1,2,\ldots, n\}$
and an $[n]$-group is a group $G$ equipped with a distinguished collection of
subgroups $\Set{$G_I$}{$I\subseteq [n]$}$ such that $G_I\subseteq G_J$ whenever
$I\subseteq J$.  The groups $\AutFn$ and $\Mod_n^1$ along with their subgroups
$\IA_n$ and $\Torelli_n^1$ can be endowed with an $[n]$-group
structure (see Definitions~\ref{definition:ngroupautfn}~and~\ref{definition:ngroupmodg}); 
indeed, this was essentially done by Church and Putman~\cite{CP}, though the technical setup
of that paper is different from ours.

\para{Weakly commuting}
The key property we shall exploit is that these $[n]$-group structures
are {\em weakly commuting}.  By definition, an $[n]$-group $G$ is weakly commuting if for all disjoint $I,J \subseteq [n]$, there exists
some $g \in G$ such that the subgroups $G_I$ and $(G_J)^g = g^{-1} G_J g$ commute.
A closely related (but different) notion of a {\em partially commuting} $[n]$-group played an important role in \cite{EH}. We also note that weakly commuting $[n]$-groups are unrelated to the ``weak FI-groups'' that appeared in \cite{CP}, despite the similar terminology. We will not use FI-groups or weak FI-groups in this paper.

\para{BNS invariant}
Let $G$ be a finitely generated group.  
The BNS invariant is a powerful tool for studying the finite generation of subgroups 
of $G$ that contain the commutator subgroup $[G,G]$.
Let $\Hom(G,\R)$ denote the set of additive characters of
$G$, that is, homomorphisms from $G$ to $(\R,+)$. Let $S(G)$ denote the set consisting of nonzero characters of $G$ modulo multiplication by positive scalars.  As a topological space, this set is a sphere of dimension $b_1(G)-1$, 
where $b_1(G)$ is the first Betti number of $G$.
Bieri, Neumann, and Strebel \cite{BieriNeumannStrebel} introduced a certain subset $\Sigma(G)$ of $S(G)$, now called the \emph{BNS invariant of $G$}, that completely determines which subgroups of $G$ containing $[G,G]$ are finitely generated. The larger $\Sigma(G)$ is, the more such subgroups are finitely generated; in particular, all of them (including $[G,G]$ itself) are finitely generated if and only if $\Sigma(G)=S(G)$.

\para{Commuting elements}
As we will make precise in Lemma~\ref{lemma:BNS-EH} below, the presence of large
numbers of commuting generators for a group can force $\Sigma(G)$ to be a very large
subset of $S(G)$. 
This sort of mechanism has been used to completely determine the BNS invariant for several important classes of groups, including right-angled Artin groups \cite{MeierVanWyk} and pure braid groups \cite{KMM-BNS}.
However, this mechanism is usually insufficient by itself to show that $\Sigma(G) = S(G)$ (as it must
be since in the aforementioned classes of groups 
the commutator subgroup $[G,G]$ is never finitely generated
except when it is trivial).  The obstacle is the existence of nonzero
characters that vanish on almost all generators involved in these commutation relations.

\para{Computing the BNS invariant of Torelli groups}
In \cite{EH}, it was proved that $\Sigma(\Torelli_n^1)=S(\Torelli_n^1)$ for $n \geq 12$ and that
$\Sigma(\IA_n) = S(\IA_n)$ for $n \geq 4$.
  The proof
was based on the following two properties of the groups $G = \Torelli_n^1$ and $G = \IA_n$:
\begin{compactitem}
\item Similarly to right-angled Artin groups, $G$ has a finite generating set in which many pairs of elements commute.
\item The group $G$ also has a large group of outer automorphisms coming from conjugation by $\Mod_n^1$ and $\AutFn$, respectively.
\end{compactitem}
In particular, the outer automorphism group of $G$ contains a natural copy of $\SL_n(\Z)$. The corresponding action of $\SL_n(\Z)$ on $G/[G,G]$ induces an action on $S(G)$ that is ``sufficiently mixing'', which implies in particular that every orbit on $S(G)$ contains characters that do not vanish on large
numbers of generators.  Combining this with the fact that $\Sigma(G)$ is invariant under $\Out(G)$,
one deduces the equality $\Sigma(G)=S(G)$.

\para{Generalization to higher terms of the lower central series}
To extend the finite generation of $[G,G]$ for $G = \Torelli_n^1$ or $G = \IA_n$ to $\gamma_k G$ with $k>2$, we adopt an inductive approach. Assuming by induction that $\gamma_k G$ is finitely generated,
the Bieri--Neumann--Strebel criterion (see Theorem~\ref{theorem:BNS}) says that to prove this for $\gamma_{k+1}G$ we must show that $\Sigma(\gamma_k G)$ contains every character vanishing on $\gamma_{k+1}G$, or in other words that $\Sigma(\gamma_k G)$ contains
the entire sphere $S(\gamma_k G/\gamma_{k+1}G)$.

If one attempts to use the method of \cite{EH} inductively, the following issue arises.
In order to apply this method to an $[n]$-group $G$, one needs to know that $G$ has a finite generating set consisting of elements of ``small complexity''. Even if $G$ possesses such a generating set, it is impossible to deduce the same for $[G,G]$ using the Bieri--Neumann--Strebel theorem, as the proof of the latter is inherently ineffective. We resolve this issue with two ideas:
\begin{compactitem}
\item The first is the notion of the {\em commuting graph} of an $[n]$-group 
(see Definition~\ref{def:commutinggraph}).  We will use commuting graphs to show that the following holds for $n \gg k$:
if $\gamma_k G$ is finitely generated, then we can find a ``nice'' generating set for $\gamma_k G$ which has enough commuting elements. The latter implies that $\Sigma(\gamma_k G)$ contains a large open subset of  $S(\gamma_k G)$.
\item The second provides a way to take this open subset and use it to show that $\Sigma(\gamma_k G)$ contains all of $S(\gamma_k G/\gamma_{k+1}G)$.
 Similarly to \cite{EH}, this part of the proof uses the action of $\Out(G)$ on $S(\gamma_k G/\gamma_{k+1}G)$; however, instead of using combinatorial properties of this action as in \cite{EH}, we will give an abstract argument involving algebraic geometry. This aspect of our proof is reminiscent of \cite{DimcaPapadimaKg} and \cite{PapadimaSuciuKg}, but unlike those two papers we will only need very basic facts from algebraic geometry.
\end{compactitem}

\enlargethispage{\baselineskip}
\para{Outline}
In the short \S\ref{section:bns} we record the properties of the BNS invariant that we will use.  Next, in \S\ref{section:Kg} we will prove Theorem~\ref{maintheorem:johnsonkernel}. 
This proof foreshadows in a simplified setting many of the ideas used in the remainder of the paper. 
In \S\ref{section:preliminaries} we introduce   the technical framework we will use for the rest of the paper.
We use this framework in \S\ref{section:maintools} to prove a general result that will imply our main theorems, and finally in \S\ref{section:mainproofs} we prove those theorems.

\para{Conventions}
Let $G$ be a group.  For $g,h \in G$, we write $h^g = g^{-1} h g$ and $[g,h] = g^{-1} h^{-1} g h$.  Also, for a subgroup
$H \subseteq G$ and $g \in G$ we write $H^g = g^{-1} H g$.

\section{Preliminaries on the BNS invariant}
\label{section:bns}

Let $G$ be a finitely generated group.  This section contains a brief introduction
to the BNS invariant $\Sigma(G)$ of $G$; see \cite{StrebelBook} for a reference
that proves all statements for which we do not provide references.
Recall from the introduction that $\Sigma(G)$ is a subset of $S(G)$, where
\[S(G)=(\Hom(G,\R)\setminus\{0\})/\R^\times_+\]
is the quotient of the set
of nonzero characters $\Hom(G,\R)\setminus\{0\}$ by the equivalence relation that identifies two characters
if they differ by multiplication by a positive scalar. For a nonzero $\chi \in \Hom(G,\R)$, write
$[\chi]$ for its image in $S(G)$.  There are many equivalent
ways to define $\Sigma(G)$.  Perhaps the easiest to state involves the connectedness
of certain subgraphs of the Cayley graph of $G$.

\begin{definition}
Let $G$ be a finitely generated group with a fixed finite generating set $S$.  Let $\cC(G,S)$ be the Cayley graph
of $G$ with respect to $S$.  Given $[\chi]\in S(G)$ represented by $\chi\in \Hom(G,\R)$, 
the \emph{BNS invariant} $\Sigma(G)$ is defined as the set of all $[\chi] \in S(G)$ such that the full subgraph of $\cC(G,S)$ spanned by the set $\Set{$g \in G$}{$\chi(g) \geq 0$}$ is connected.
\end{definition}

This definition does not depend on the choice of $S$ (though this is not obvious).
Thinking of $\Aut(G)$ as acting on $G$ on the right, we obtain an action of $\Aut(G)$ on $\Hom(G,\R)$ via the formula
\[(\alpha\cdot \chi)(g) = \chi(g^\alpha) \quad \quad (\alpha \in \Aut(G), \chi \in \Hom(G,\R), g \in G).\]
This descends to an action of $\Aut(G)$ on $S(G)$ which factors through $\Out(G)$.  The fact that the BNS invariant is independent of the generating set
implies that $\Sigma(G)$ is invariant under this action.

If $N$ is a normal subgroup of $G$, then we can identify
$\Hom(G/N,\R)$ with the subset of $\Hom(G,\R)$ consisting of those characters that vanish on
$N$.  This induces an identification of $S(G/N)$ with a subset of $S(G)$.
When $N$ contains $[G,G]$, the following theorem of
Bieri, Neumann and Strebel characterizes finite generation of $N$ in terms of $\Sigma(G)$.

\enlargethispage{\baselineskip}
\begin{theorem}[{\cite[Theorem~B1]{BieriNeumannStrebel}}]
\label{theorem:BNS}
Let $G$ be a finitely generated group and $N$ be a subgroup of $G$ containing $[G,G]$.
Then $N$ is finitely generated if and only if
$S(G/N) \subseteq \Sigma(G)$.
\end{theorem}

The following sufficient condition for an element of $S(G)$ to lie in $\Sigma(G)$ was established by Ershov and He.

\begin{lemma}[{\cite[Proposition 2.4(b)]{EH}}]
\label{lemma:BNS-EH}
Let $G$ be a finitely generated group and let $\chi\in \Hom(G,\R)$ be a nonzero character.
Suppose there exists a finite sequence $x_1,\ldots, x_r$ of elements of $G$ such that the following
hold.
\begin{compactenum}[label=(\roman*)]
\item\label{part:EHlemma-gen} $G$ is generated by $x_1,\ldots,x_r$.
\item\label{part:EHlemma-x1} $\chi(x_1)\neq 0$.
\item\label{part:EHlemma-commutator} For every $2\leq i\leq r$, there exists $j<i$ such that $\chi(x_j)\neq 0$ and such that the commutator $[x_j,x_i]$ lies in the subgroup
generated by $x_1,\ldots, x_{i-1}$.
\end{compactenum}
Then $[\chi] \in \Sigma(G)$.
\end{lemma}

\begin{remark}
\label{remark:specialcase}
An important special case of Lemma~\ref{lemma:BNS-EH} is when $x_j$ and $x_i$ in \ref{part:EHlemma-commutator}  are required to commute. This special case of the lemma was known prior to \cite{EH}.  It is essentially equivalent to \cite[Lemma~1.9]{KMM-BNS}, but the basic idea goes back further (compare with \cite[Theorem~6.1]{MeierVanWyk} from 1995).
This special case was sufficient for the purposes of \cite{EH} except when dealing with $\IA_4$ and $\IA_5$.  To prove Theorem~\ref{maintheorem:johnsonkernel} 
and the $k=2$ case of Theorem~\ref{maintheorem:ialcs} we only need this special case of Lemma~\ref{lemma:BNS-EH}, but for Theorems~\ref{maintheorem:torellilcs}--\ref{maintheorem:iajohnson}
we will make essential use of the full strength of Lemma~\ref{lemma:BNS-EH}.
\end{remark}

\section{Proof of Theorem~\ref{maintheorem:johnsonkernel}}
\label{section:Kg}

In this section we will prove Theorem~\ref{maintheorem:johnsonkernel}, 
which asserts that any subgroup of $\Torelli_g^b$ containing $[\Torelli_g^b,\Torelli_g^b]$ 
is finitely generated for $g \geq 4$ and $b \in \{0,1\}$.  
This will imply in particular that
$\JohnsonKer_g^b$ is finitely generated.
This proof follows the same outline as the proofs of our other results, but avoids a lot of
technicalities. For that reason, we suggest that the reader begin with this section.

\para{BP graph}
We will need a certain graph constructed from elements of the Torelli group.
A {\em genus-$1$ bounding pair} on $\Sigma_g^1$ (often shortened to a {\em genus-$1$ BP}) is an ordered pair $(x,y)$ of 
disjoint homologous nonseparating simple closed curves on $\Sigma_g^1$ whose union $x \cup y$ separates $\Sigma_g^1$ 
into two subsurfaces, one homeomorphic to $\Sigma_1^2$ and the other to $\Sigma_{g-2}^3$ (see Figure~\ref{figure:bpgraph}). 
If $(x,y)$ is a genus-$1$ BP, then the corresponding product of Dehn twists $T_x T_y^{-1}\in \Torelli_g^1$ is called 
a {\em genus-$1$ BP map}.
This gives a bijection between genus-$1$ BP maps and isotopy classes of genus-$1$ BPs.
All genus-$1$ BPs on $\Sigma_g^1$ lie in the same $\Mod_g^1$-orbit 
(see \cite[\S 1.3]{FarbMargalitPrimer}), and therefore all genus-$1$ BP maps are conjugate in $\Mod_g^1$.

\begin{definition}
Let $\BP{g}$ denote the {\em genus-1 BP graph} whose vertices are genus-$1$ BP maps in $\Torelli_g^1$, and where two elements are connected by an edge if they commute.
\end{definition}

\begin{remark}
Given genus-$1$ BPs $(x,y)$ and $(x',y')$, if we can  homotope the curves such that
$x \cup y$ is disjoint from $x' \cup y'$ (see Figure~\ref{figure:bpgraph}), then the BP maps $T_x T_y^{-1}$ and
$T_{x'} T_{y'}^{-1}$ commute.  The converse is also true, though we will not actually need this.
\end{remark}

\begin{proposition}
\label{prop:BPgraph}
The genus-1 BP graph $\BP{g}$ is connected for $g\geq 4$.
\end{proposition}

Proposition~\ref{prop:BPgraph} is likely folklore, but we do not know a reference, so we include a proof.

\begin{proof}[Proof of Proposition~\ref{prop:BPgraph}]
We will use the main idea from \cite[Lemma~2.1]{PutmanConnectivityNote}.  
Let ${\varphi=T_xT_y^{-1}}$ be the vertex of $\BP{g}$ depicted in Figure
\ref{figure:bpgraph}.  We must prove that there is a path in $\BP{g}$ between $\varphi$ and any other
vertex of $\BP{g}$.  The group $\Mod_g^1$ acts on $\BP{g}$ by conjugation, and this action is transitive on vertices since all genus-1 BP maps are conjugate in $\Mod_g^1$.  It is thus enough to prove that for all $f \in \Mod_g^1$, there
is a path in $\BP{g}$ from $\varphi$ to $\varphi^f$.

\Figure{figure:bpgraph}{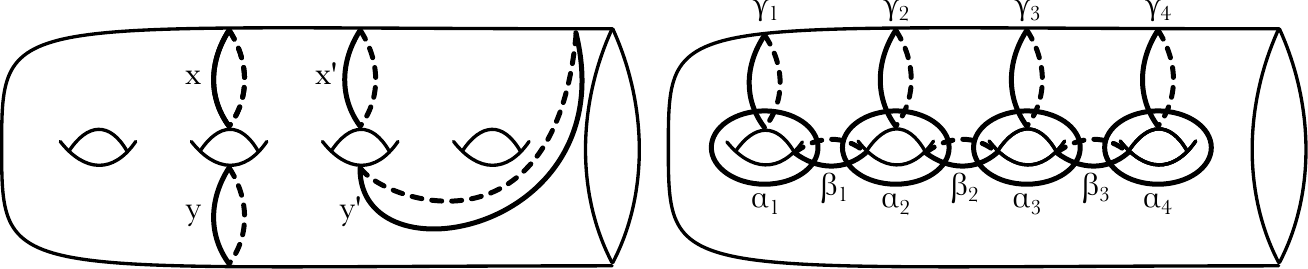}{On the left are two genus-$1$ BP maps $\varphi=T_xT_y^{-1}$ and $\psi=T_{x'}T_{y'}^{-1}$ that are adjacent in $\BP{g}$. On the right are the Dehn twists involved in the definition of the set $S$.  Note that $x'$ and $y'$ on the left are  disjoint from $\alpha_2$, so $\psi=T_{x'}T_{y'}^{-1}$ commutes with $T_{\alpha_2}$.}{78}

We begin by proving a special case of this.  Let
\begin{equation}
\label{eq:Humphries-generators}
S = \Set{$T_{\alpha_i}^{\pm 1}$}{$1 \leq i \leq g$} \cup \Set{$T_{\gamma_i}^{\pm 1}$}{$1 \leq i \leq g$} \cup
\Set{$T_{\beta_i}^{\pm 1}$}{$1 \leq i \leq g-1$}
\end{equation}
be the Dehn twists depicted in Figure~\ref{figure:bpgraph}.
The set $S$ generates $\Mod_g^1$; see \cite[Theorem 1]{JohnsonFinite}.  We claim that for all $s \in S$, 
there exists a path $\eta_s$ in $\BP{g}$ from $\varphi$ to $\varphi^s$.
Indeed, all the curves $\alpha_i$, $\gamma_i$, and $\beta_i$ are disjoint from $x$ and $y$, with the exception of $\alpha_2$. Therefore for $s \notin \{T_{\alpha_2}^{\pm 1}\}$, the map $s$ commutes with $\varphi$, so $\varphi=\varphi^s$ and the claim is trivial.  
If $s \in \{T_{\alpha_2}^{\pm 1}\}$, then letting $\psi=T_{x'}T_{y'}^{-1}$ be the BP map depicted in Figure
\ref{figure:bpgraph}, we see that $x'$ and $y'$ are disjoint from $\alpha_2$. This implies that $\psi$ commutes with both $\varphi$ and $s$, and thus also with $\varphi^{s}$. We can therefore take  $\eta_s$ to be the length 2 path from $\varphi$ to $\psi$ to $\varphi^{s}$.

We now prove the general case.  Consider $f \in \Mod_g^1$.  Write
\[f = s_1 s_2 \cdots s_{\ell} \quad \quad \text{with $s_i \in S$}.\]
For $h \in \Mod_g^1$ and $s \in S$, the path $(\eta_s)^h$ goes from $\varphi^h$ to $\varphi^{sh}$.
Letting $\bullet$ be the concatenation product on paths, the desired path
from $\varphi$ to $\varphi^f$ is then
\[\eta_{s_{\ell}} \bullet (\eta_{s_{\ell-1}})^{s_{\ell}} \bullet (\eta_{s_{\ell-2}})^{s_{\ell-1} s_{\ell}} \bullet \cdots \bullet (\eta_{s_{1}})^{s_2 s_3 \cdots s_{\ell}}.\qedhere\]
\end{proof}

\begin{remark}
The genus-1 BP graph $\BP{g}$ is not connected for $g=3$. This can be seen by noting that if $T_xT_y^{-1}$ is connected to $T_{x'}T_{y'}^{-1}$ in $\BP{3}$, then the curves $\{x,y,x',y'\}$ all share the same homology class. Therefore BP maps with different homology classes must lie in different components of $\BP{3}$. The connectivity of $\BP{g}$ is the only place in the proof of Theorem~\ref{maintheorem:johnsonkernel} where
we will use the assumption $g \geq 4$.
\end{remark}

\para{Work of Johnson}
We will need three important results of Johnson that are summarized in the following
theorem; the three parts are proved in \cite{JohnsonHomeo}, \cite{JohnsonFinite}, and \cite{JohnsonAbel} respectively.

\begin{theorem}[Johnson]
\label{theorem:johnson}
For $g \geq 3$, the following hold:
\begin{compactenum}[label=(\alph*)]
\item\label{part:johnson:BP} The group $\Torelli_g^1$ is generated by genus-1 BP maps.
\item\label{part:johnson:fg}The group $\Torelli_g^1$ is finitely generated.
\item\label{part:johnson:H1} There is a $\Mod_g^1$-equivariant isomorphism
$(\Torelli_g^1)^{\ab} \otimes \R \iso \bwedge^3 \HH_1(\Sigma_g^1;\R)$,
where $\Mod_g^1$ acts on $(\Torelli_g^1)^{\ab}$ via its conjugation action on $\Torelli_g^1$.
\end{compactenum}
\end{theorem}

\para{Zariski topology on the mapping class group}
The final preliminary ingredient we will need is a certain topology on the mapping class group.
The conjugation action of $\Mod_g^1$ on its normal subgroup $\Torelli_g^1$ induces
an action of $\Mod_g^1$ on the vector space $\Hom(\Torelli_g^1,\R)$.  This gives a group homomorphism
\[\Mod_g^1 \longrightarrow \GL(\Hom(\Torelli_g^1,\R)).\]
Endow $\GL(\Hom(\Torelli_g^1,\R))$ with the Zariski topology.
Let the {\em $\Hom(\Torelli_g^1,\R)$-Zariski topology} on $\Mod_g^1$ be the topology pulled back from the Zariski topology on $\GL(\Hom(\Torelli_g^1,\R))$ by this map.
The key property of this topology that we will use is as follows.
Recall that a topological space is irreducible if it cannot be written as the union of two proper closed subspaces.

\begin{lemma}
\label{lemma:zariskimodirred}
For $g \geq 3$, the group $\Mod_g^1$ equipped with the $\Hom(\Torelli_g^1,\R)$-Zariski topology is an
irreducible topological space.
\end{lemma}
\begin{proof}
In our proof, we will use three basic properties of irreducible spaces:
\begin{compactenum}[label=(\roman*)]
\item\label{part:pullback} if $Y$ is an irreducible topological space and $X \rightarrow Y$ is a set map, then the pullback
to $X$ of the topology on $Y$ makes $X$ into an irreducible topological space;
\item\label{part:image} if $Y \rightarrow Z$ is a continuous map between topological spaces and $Y$ is irreducible, then
the image of $Y$ in $Z$ is irreducible.
\item\label{part:subspace} a subspace $Z$ of a topological space $W$ is irreducible if and only if the closure of $Z$ in $W$ is irreducible.
\end{compactenum}
By property \ref{part:pullback} above, it is enough to prove that the image of $\Mod_g^1$ in $\GL(\Hom(\Torelli_g^1,\R))$ is an irreducible subspace.  Set $H = \HH_1(\Sigma_g^1;\R)$ and
$H_{\Z} = \HH_1(\Sigma_g^1;\Z)$.
Theorem~\ref{theorem:johnson}\ref{part:johnson:H1}  gives a $\Mod_g^1$-equivariant isomorphism
\[(\Torelli_g^1)^{\ab}\otimes\R \iso \bwedge^3 H.\]
Recall from the introduction that the algebraic intersection form on $H$ turns it into a symplectic vector space
and that the image of $\Mod_g^1$ in $\GL(H)$ is equal to $\Sp(H_{\Z})$.  The action of $\Mod_g^1$ on $H$ thus factors through a representation of the symplectic group $\Sp(H)$.
Since $H$ is a self-dual representation of $\Sp(H)$, this implies
that there is a $\Mod_g^1$-equivariant isomorphism
\[\Hom(\Torelli_g^1,\R) \iso \bwedge^3 H\]
as well.  Under this identification, the image of
$\Mod_g^1$ in $\GL(\Hom(\Torelli_g^1,\R))$ is equal to the image of $\Sp(H_{\Z})$ under the natural mapping
$\iota\colon \GL(H) \to \GL(\bwedge^3 H)$.
It is classical that the Zariski closure of $\Sp(H_{\Z})$ in $\GL(H)$ is $\Sp(H)\iso \Sp_{2g}(\R)$, which is a connected algebraic group and
hence an irreducible topological space (see, e.g., \cite[Theorem III.2.1]{CarterSegalMacdonald}).
Property~\ref{part:subspace} above thus implies that $\Sp(H_{\Z})$ is an irreducible topological
space.  Since $\iota$ is Zariski-continuous, property~\ref{part:image} above implies that
$\iota(\Sp(H_{\Z}))$ is irreducible, as desired.  We remark that the self-duality of $\bigwedge^3 H$ is not essential here; we could apply the exact same argument to the natural
map $\GL(H)\to \GL((\bigwedge^3 H)^*)$.
\end{proof}

\para{Putting it all together}
We now prove Theorem~\ref{maintheorem:johnsonkernel}.

\begin{proof}[Proof of Theorem~\ref{maintheorem:johnsonkernel}]
Fix some $g \geq 4$.
As discussed in \S \ref{section:outlineproof}, it suffices to prove that $[\Torelli_g^1,\Torelli_g^1]$ is finitely generated,
which by Theorem~\ref{theorem:BNS} is equivalent to showing that the
BNS invariant $\Sigma(\Torelli_g^1)$ is all of $S(\Torelli_g^1)$.
 
Theorem~\ref{theorem:johnson}\ref{part:johnson:BP} and \ref{part:johnson:fg} tell us that for $g \geq 3$
there is a finite set of genus-$1$ BP maps that generate $\Torelli_g^1$; we emphasize that we do \emph{not} need to know such a finite generating set explicitly.
Proposition~\ref{prop:BPgraph} says that $\BP{g}$ is connected for $g \geq 4$.  Combining
these two facts, we see that there exists a finite set $\Lambda=\{\lambda_1,\ldots,\lambda_r\}$ of genus-$1$ BP maps
that  generates $\Torelli_g^1$ such that the full subgraph of $\BP{g}$ spanned by $\Lambda$ is connected (simply begin with a generating set and add more vertices until the subgraph is connected).

The key property of $\Lambda$ is as follows:
\begin{equation}
\label{eqn:lambdakey}
\text{If $\rho\colon \Torelli_g^1\to \R$ satisfies $\rho(\lambda) \neq 0$ for all $\lambda \in \Lambda$, then $[\rho]\in \Sigma(\Torelli_g^1)$}.
\end{equation}
Indeed, consider such a $\rho$.
We can assume that $\Lambda=\{\lambda_1,\ldots,\lambda_r\}$ is enumerated in increasing order of distance
from some fixed basepoint in the subgraph spanned by $\Lambda$. This guarantees that for all $1 < i \leq r$, there exists some $1 \leq j < i$ such that $\lambda_i$ and $\lambda_j$ are adjacent in $\BP{g}$. In other words, the genus-$1$ BP maps $\lambda_i$ and $\lambda_j$ commute. Since $\rho(\lambda_i) \neq 0$ for all $1 \leq i \leq r$, the sequence $\lambda_1,\ldots,\lambda_r$ satisfies
all three conditions of Lemma~\ref{lemma:BNS-EH}.  That lemma now implies that $[\rho]\in \Sigma(\Torelli_g^1)$.
Note that here we only used the special case of Lemma~\ref{lemma:BNS-EH} discussed in Remark \ref{remark:specialcase}. 

Now consider an arbitrary nonzero $\rho \colon \Torelli_g^1\to \R$; we will prove that $[\rho]\in \Sigma(\Torelli_g^1)$. For each $\lambda \in \Lambda$, define
\[\cZ_{\lambda}=\Set{$\gamma\in \Mod_g^1$}{$(\gamma\cdot \rho)(\lambda)=0$}.\]
For a fixed $\lambda$ the condition $\varphi(\lambda)=0$ is a Zariski-closed condition on $\varphi$, so each $\cZ_{\lambda}$ is a closed subspace of $\Mod_g^1$ in the $\Hom(\Torelli_g^1,\R)$-Zariski topology. 
Moreover, we claim that each $\cZ_{\lambda}$ is a
proper subset of $\Mod_g^1$.  To see this, observe that the equality $\cZ_{\lambda} = \Mod_g^1$ would mean that
$(\gamma\cdot \rho)(\lambda)=\rho(\lambda^{\gamma})$ vanishes for all $\gamma\in \Mod_g^1$. Since $\lambda$ is a genus-$1$ BP map and all genus-$1$ BP maps are conjugate, this would mean that $\rho$ vanishes on all genus-$1$ BP maps.  But by Theorem~\ref{theorem:johnson}\ref{part:johnson:BP}, $\Torelli_g^1$ is generated by genus-$1$ BP maps for $g\geq 3$, so no nonzero homomorphism $\rho\colon \Torelli_g^1\to \R$ can vanish on all these elements. This verifies that $\cZ_{\lambda}\subsetneq \Mod_g^1$ as claimed.

Lemma \ref{lemma:zariskimodirred} says that $\Mod_g^1$ is an irreducible space with respect to the $\Hom(\Torelli_g^1,\R)$-Zariski topology.  Since an irreducible space cannot be written as a finite union of closed proper subspaces, 
we deduce that 
\[\bigcup_{\lambda \in \Lambda} \cZ_{\lambda} \subsetneq \Mod_g^1.\] 
Choose some $\gamma \in \Mod_g^1$ such that $\gamma \notin \cZ_{\lambda}$ for all $\lambda \in \Lambda$.  By definition, this means
that $(\gamma\cdot \rho)(\lambda)\neq 0$ for all $\lambda\in \Lambda$.  Applying \eqref{eqn:lambdakey} to $\gamma\cdot \rho$, we deduce that
$[\gamma\cdot \rho]\in \Sigma(\Torelli_g^1)$.
Since $\Sigma(\Torelli_g^1)$ is invariant under automorphisms of $\Torelli_g^1$, it follows that 
$[\rho]\in \Sigma(\Torelli_g^1)$ as well. Since $\rho$ was arbitrary, this shows that 
$\Sigma(\Torelli_g^1)$ is all of $S(\Torelli_g^1)$, as desired.
\end{proof}

We will use the same exact approach in \S\ref{section:ialcs} to prove that $[\IA_n,\IA_n]$ is finitely generated for $n\geq 4$. For now, we record the structure of the above argument; the proof of the following theorem follows exactly the proof of Theorem~\ref{maintheorem:johnsonkernel} above.

\begin{theorem}
\label{thm:Kg-argument-general}
Let $G$ be a finitely generated group. Suppose that a group $\Gamma$ acts on $G$ by automorphisms such that the following hold.
\begin{compactenum}
\item The group $G$ is generated by a single $\Gamma$-orbit $C\subset G$.
\item The image of $\Gamma$ in $\GL(\Hom(G,\R))$ is irreducible in the Zariski topology.
\item The graph whose vertices are elements $c\in C$ where two elements are connected by an edge if they commute is connected.
\end{compactenum}
Then $[G,G]$ is finitely generated.
\end{theorem}

\section{Basic properties of \texorpdfstring{$\boldsymbol{[n]}$}{[n]}-groups}
\label{section:preliminaries}

This section contains preliminary definitions and results that will be used in the technical framework of the remainder of the paper.
  It has three sections: \S\ref{section:ngroups} introduces
$[n]$-groups and the two fundamental examples $\Aut(F_n)$ and $\Mod_g^1$, then \S\ref{section:centralseries} discusses central series and their Lie algebras, and finally
\S\ref{section:zariski} discusses Zariski-irreducible actions.

\subsection{\texorpdfstring{$\boldsymbol{[n]}$}{[n]}-groups}
\label{section:ngroups}

Set $\N = \{1,2,3,\ldots\}$.  For $n \in \N$, define $[n] = \{1,\ldots,n\}$.

\begin{definition}
Let $n\in\N$.  An {\em $[n]$-group} is a group $G$ equipped with a distinguished collection of subgroups $\Set{$G_I$}{$I \subseteq [n]$}$ such that
$G_{[n]}=G$ and such that $G_I\subseteq G_J$ whenever $I\subseteq J$.  We say that an $[n]$-group $V$ is an 
{\em $[n]$-vector space} if $V$ is a vector space and each $V_I$ is a subspace.
\end{definition}

Any subgroup or quotient of an $[n]$-group naturally inherits the structure of an $[n]$-group as follows.
\begin{definition}
\label{definition:ngroupsubquotient}
Let $G$ be an $[n]$-group.  For a subgroup $H$ of $G$, define an $[n]$-group structure on $H$ by setting
$H_I = H \cap G_I$ for all $I \subseteq [n]$.  For a quotient $G/K$ of $G$, define an $[n]$-group structure on $G/K$
by setting $(G/K)_I = G_I K / K$ for all $I \subseteq [n]$.
\end{definition}

In addition, any abelian $[n]$-group $A$ can be tensored with $\R$ to
obtain an $[n]$-vector space as follows:

\begin{definition}
\label{definition:nvectorspace}
Let $A$ be an abelian $[n]$-group. Then $A\otimes \mathbb R$ becomes an $[n]$-vector space by setting
$(A\otimes \mathbb R)_I=A_I\otimes \mathbb R$ for all $I \subseteq [n]$.
\end{definition}

\para{Key examples}
Before moving on, we define the two key examples of $[n]$-groups that we will use in this paper, namely $\AutFn$ and the
mapping class group.  The structure on $\AutFn$ is easy to define.

\begin{definition}
\label{definition:ngroupautfn}
Let $F_n$ be the free group on $\{x_1,\ldots,x_n\}$ and let $\Gamma=\AutFn$.
For
$I \subseteq [n]$, set
$F_I = \lSet{$x_i$}{$i \in I$}$, and define
\[\Gamma_I = \Set{$f \in \Gamma$}{$f(x_i) \in F_I$ for all $i \in I$ and $f(x_j) = x_j$ for all $j \in [n] \setminus I$}.\]
This endows $\Gamma$ with the structure of an $[n]$-group.
\end{definition}

For the mapping class group, it is a bit more subtle.  The natural measure of complexity for the mapping class
group is the genus $g$, so these will be $[g]$-groups.  The starting point is the following lemma (which is
implicit in \cite[\S 4.1]{CP}).

\Figure{figure:bigsurface}{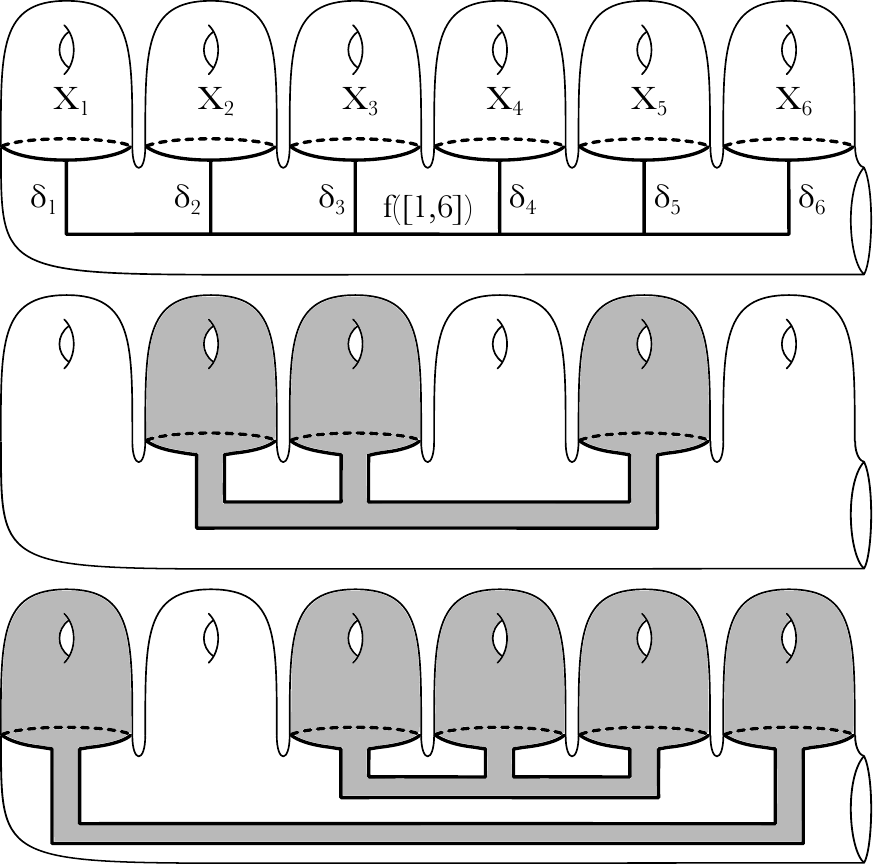}{On the top are the $X_i$ and the $\delta_i$ and $f([1,g])$ for $g=6$.  In the middle is
$\Sigma_{\{2,3,5\}}$, and on the bottom is $\Sigma_{\{3,4,5\}}$ isotoped to be disjoint from $\Sigma_{\{1,6\}}$.}{78}

\begin{lemma}
\label{lemma:choosesurfaces}
Fix some $g \geq 1$, and let $\Sigma = \Sigma_g^1$ be a genus $g$ surface with $1$ boundary component.  We can then
choose subsurfaces $\Sigma_I$ for each $I \subseteq [g]$ such that the following hold:
\begin{compactenum}[label=(\roman*)]
\item\label{part:genus} $\Sigma_I$ is homeomorphic to $\Sigma_{\abs{I}}^1$.
\item\label{part:sigmag} $\Sigma_{[g]}=\Sigma$.
\item\label{part:sigmaI} $\Sigma_I$ is isotopic to a subsurface of $\Sigma_J$ whenever $I\subseteq J$.
\item\label{part:disjointconsecutive} If $I,J\subseteq [g]$ are disjoint
and $I$ consists of consecutive integers, then $\Sigma_I$ is isotopic
to a subsurface disjoint from $\Sigma_J$.  We remark that this need not hold
if $I$ does not consist of consecutive integers.
\end{compactenum}
\end{lemma}
\begin{proof}
As in Figure \ref{figure:bigsurface}, choose disjoint subsurfaces $X_1,\ldots,X_g \subset \Sigma_g^1$ with $X_i \cong \Sigma_1^1$
for each $i$.  Let $[1,g] \subset \R$ denote the closed interval.  Letting 
\[Y = \Sigma_g^1 \setminus \bigcup_{i=1}^g \Interior(X_i),\]
choose an embedding $f\colon [1,g] \rightarrow \Interior(Y)$, and let $\delta_1,\ldots,\delta_g$ be arcs in $Y$ such that
$\delta_i$ connects a point on $\partial X_i$ to $f(i)$.  Pick the $\delta_i$ such that they all approach $f([1,g])$ from the same side, 
such that they are all pairwise disjoint from each
other, and such that each $\delta_i$ only intersects $\partial Y \cup f([1,g])$ at its endpoints.  For $I \subseteq [g]$
enumerated as
\[I = \{i_1 < \cdots < i_k\},\]
let $\Sigma_I$ be a closed regular neighborhood of
\[f([i_1,i_k]) \cup \left(\bigcup_{j=1}^k \delta_{i_j}\right) \bigcup \left(\bigcup_{j=1}^k X_{i_j}\right).\]
See Figure \ref{figure:bigsurface}.  It is clear that these subsurfaces satisfy the conclusions of the lemma.  The
only nontrivial one is conclusion \ref{part:disjointconsecutive}, which is illustrated at the bottom of Figure \ref{figure:bigsurface}.
\end{proof}

This allows us to make the following definition:

\begin{definition}
\label{definition:ngroupmodg}
Let $\Sigma = \Sigma_g^1$ be a genus $g$ surface with $1$ boundary component.  We define a $[g]$-group structure on
$\Gamma = \Mod(\Sigma) = \Mod_g^1$ as follows.
Fix once and for all a collection of subsurfaces $\Sigma_I$ as in Lemma \ref{lemma:choosesurfaces}, and for each $I\subseteq [g]$ define $\Gamma_I$ to be the subgroup of $\Gamma$ consisting of mapping classes supported on $\Sigma_I$.
Conclusions \ref{part:sigmag} and \ref{part:sigmaI} of Lemma \ref{lemma:choosesurfaces} imply that $\Gamma$ is an $[g]$-group.
\end{definition}

\begin{remark}
Since any subgroup of an $[n]$-group inherits an $[n]$-group structure, Definition~\ref{definition:ngroupautfn} induces an $[n]$-group structure
on $\IA_n$. Similarly, Definition~\ref{definition:ngroupmodg} induces an $[g]$-group structure on $\Torelli_g^1$.
\end{remark}

\para{Fundamental properties} We now turn to some fundamental properties of $[n]$-groups.

\begin{definition}
Let $G$ be an $[n]$-group.  We say that $G$ 
is \emph{generated in degree $d$} if $G$ is generated by the set $\Set{$G_I$}{$I \subseteq [n]$, $\abs{I}=d$}$.
We write $d(G)$ for the smallest $d\geq 0$ such that $G$ is generated in degree $d$.
\end{definition}

\begin{remark}
As we will explain in \S\ref{section:mainproofs}, the $[n]$-group $\AutFn$ is generated in degree $2$ while its subgroup $\IA_n$ is generated in degree $3$.  Similarly, the $[g]$-group $\Mod_g^1$ is generated in degree $2$ and its subgroup $\Torelli_g^1$ is generated in degree $3$.
\end{remark}

\begin{definition}
\label{definition:threeproperties}
Let $G$ be an $[n]$-group.  
\begin{compactenum}[label=(\roman*)]
\item\label{part:transitive} We say that $G$ is \emph{transitive} if $G_I$ and $G_J$ are conjugate in $G$ for all $I,J \subseteq [n]$ with $\abs{I}=\abs{J}$.
\item\label{part:commuting} We say that $G$ is \emph{commuting} if $G_I$ and $G_J$ commute for all disjoint $I,J \subseteq [n]$.
\item\label{part:weaklycommuting} We say that $G$ is \emph{weakly commuting} if for all disjoint $I,J \subseteq [n]$, there exists $g \in G$ such that
$(G_I)^g = g^{-1} G_I g$ and $G_J$ commute.
\end{compactenum}
\end{definition}

\begin{remark}
\label{remark:basicproperties}
We can see directly from Definition~\ref{definition:ngroupautfn} that the $[n]$-group $\AutFn$ is commuting, and thus so is its subgroup  $\IA_n$. It is easy to see that $\AutFn$
is transitive by using the automorphisms permuting the generators, but one can show that $\IA_n$ is not transitive (we
omit the proof since this will not be needed).  In Lemma \ref{lemma:torellibasic} below, we will see that $\Mod_g^1$ is also transitive and that $\Mod_g^1$ and $\Torelli_g^1$ are weakly commuting, but not commuting.
\end{remark}

\begin{remark}
If $G$ is an $[n]$-group satisfying any of the properties \ref{part:transitive}--\ref{part:weaklycommuting} in Definition~\ref{definition:threeproperties}, then the same is true for any quotient of $G$. If
$H$ is a subgroup of a commuting $[n]$-group $G$, then the $[n]$-group $H$ is also commuting. However, the properties of being transitive or weakly commuting need not pass to subgroups since their definitions refer to conjugation by elements of $G$ that need not lie in the subgroup.  For instance, as we mentioned in Remark \ref{remark:basicproperties} the $[n]$-group
$\Aut(F_n)$ is transitive, but its subgroup $\IA_n$ is not.
\end{remark}

\subsection{Central series and Lie algebras}
\label{section:centralseries}

We recall the following definition.

\begin{definition}
A {\em central series} of a group $G$ is a descending chain
\[G = G(1) \supseteq G(2) \supseteq G(3) \supseteq \cdots\]
of subgroups of $G$ such that $[G,G(k)] \subseteq G(k+1)$ for all $k \geq 1$.
If $G$ is a normal subgroup of a group $\Gamma$, we will call such a central series a \emph{$\Gamma$-normal central series} if $G(k)$ is normal in $\Gamma$ for all $k\geq 1$.
\end{definition}

The most common example of a central series is the lower central series of $G$; this is a $\Gamma$-normal central series if $G$ is normal in $\Gamma$.
Every central series determines a graded Lie algebra in the following way. Note that $G(k) / G(k+1)$ is an abelian group.

\begin{definition}
Let $G = G(1) \supseteq G(2) \supseteq \cdots$
be a central series of $G$.  The {\em associated graded real Lie algebra} $\cL$ is the real Lie algebra \[\cL = \bigoplus_{k=1}^{\infty} \cL(k),\quad\mbox{where}\quad\cL(k) = (G(k) / G(k+1)) \otimes \R.\] 
The Lie bracket on $\cL$ is induced by the commutator bracket on $G$, which descends to a bilinear map $\cL(k) \otimes \cL(\ell) \rightarrow \cL(k+\ell)$; see
\cite[\S II]{SerreLie}.  If $G\normal \Gamma$ and $G(k)$ is a $\Gamma$-normal central series, the conjugation action of $\Gamma$ on $G$ induces a linear action of $\Gamma$ on each $\cL(k)$; this preserves the Lie bracket and thus extends to an action of $\Gamma$ on $\cL$ by Lie algebra automorphisms.
\end{definition}

\begin{remark}
It is also common to consider the Lie ring $\bigoplus G(k)/G(k+1)$, without tensoring with $\R$. This object plays a key role in \cite{EH}, 
but in this paper we will only deal with the real Lie algebra~$\cL$.
\end{remark}

We next discuss how this interacts with an $[n]$-group structure.

\begin{definition}
\label{def:[n]vectorspacestructure}
Let $G$ be an $[n]$-group and let $G = G(1) \supseteq G(2) \supseteq  \cdots$
be a central series of $G$.  Using Definition~\ref{definition:ngroupsubquotient}, the $[n]$-group structure
on $G$ induces an $[n]$-group structure on the subquotient $G(k)/G(k+1)$. By Definition~\ref{definition:nvectorspace}, the latter induces an $[n]$-vector space structure on $\cL(k) = (G(k)/G(k+1)) \otimes \R$.  Unwinding the definitions to get an explicit description, we see that the subspace $\cL(k)_I$ is the subspace of $\cL(k)$ spanned by the image of $G(k)_I = G(k) \cap G_I$.
\end{definition}

\begin{remark}
For each $I \subseteq [n]$, this gives a Lie subalgebra $\cL_I=\bigoplus\cL(k)_I$ of $\cL$.  We warn the reader
that even if $\cL$ is generated as a Lie algebra by $\cL(1)$, the Lie algebra $\cL_I$ need not be generated
by $\cL(1)_I$.  
\end{remark}

\subsection{Zariski-irreducible actions}
\label{section:zariski}

The following definition will play an important role in our proofs.

\begin{definition}
Let $\Gamma$ be a group acting on a finite-dimensional vector space $V$.  The {\em $V$-Zariski topology} on $\Gamma$ is the pullback
to $\Gamma$ of the Zariski topology on $\GL(V)$ under the map $\Gamma \rightarrow \GL(V)$.  We say that the
action of $\Gamma$ on $V$ is {\em Zariski-irreducible} if $\Gamma$ is irreducible in the $V$-Zariski
topology (or, equivalently, if 
the image of $\Gamma$ in $\GL(V)$ is irreducible).
\end{definition}

\begin{remark}
\label{remark:dualZariski}
The topologies on $\Gamma$ obtained from the action on $V$ and on the dual vector space $V^*$ coincide.
\end{remark}

\begin{remark}
\label{remark:irreduciblerepresentation}
The notion of Zariski-irreducible action should not be confused with the unrelated concept of an irreducible representation. For example, if $V$ is a $\C$-vector space and $\lambda \in \C^{\times}$ has infinite order, then the diagonal
action of $\Z$ on $V$ defined by $n \cdot \vec{v} = \lambda^n \vec{v}$ is Zariski-irreducible, but is only irreducible as a representation if $\dim V = 1$.
\end{remark}

The main property of Zariski-irreducible actions we will need is the following lemma. We note that this lemma is closely related to Lemma 3.2 of \cite{DimcaPapadimaKg}, which was used for a similar purpose. 
\begin{lemma}
\label{lemma:zariskilie}
Suppose that $\Gamma$ acts on a graded Lie algebra $\cL=\bigoplus \cL(k)$. If $\cL$ is generated by $\cL(1)$ and the action of $\Gamma$ on $\cL(1)$ is Zariski-irreducible, then the action of $\Gamma$ on $\cL(k)$ is Zariski-irreducible for all $k\geq 1$.\end{lemma}
\begin{proof}
We first observe that the image of
$\Gamma$ in $\GL(\cL(1)^{\otimes k})$ under the diagonal action is irreducible.  Indeed, the map $\Gamma \rightarrow \GL(\cL(1)^{\otimes k})$
can be factored as
\[\Gamma \rightarrow \GL(\cL(1)) \rightarrow \GL(\cL(1)^{\otimes k}).\]
By assumption, the image of $\Gamma$ in $\GL(\cL(1))$ is irreducible.  Since $\GL(\cL(1)) \rightarrow \GL(\cL(1)^{\otimes k})$
is a continuous map, this implies that the image of $\Gamma$ in $\GL(\cL(1)^{\otimes k})$ is irreducible, as claimed.

Now consider the linear map $m\colon \cL(1)^{\otimes k} \rightarrow \cL(k)$ taking the element
$v_1 \otimes \cdots \otimes v_k$ to 
$[[[v_1,v_2],\cdots],v_k]$. Since $\cL$ is generated
by $\cL(1)$, the map $m$ is surjective. Let $W=\ker m$ be its kernel.  Since $\Gamma$ acts on $\cL$ by Lie algebra automorphisms,
the map $m\colon \cL(1)^{\otimes k} \rightarrow \cL(k)$ is $\Gamma$-equivariant. Therefore the image of $\Gamma$ in $\GL(\cL(1)^{\otimes k})$ is contained in the subgroup $\GL(\cL(1)^{\otimes k},W)$ of elements which preserve $W$.

In other words, the map $\Gamma \rightarrow \GL(\cL(k))$ can be factored as
\[\Gamma \rightarrow \GL(\cL(1)^{\otimes k},W) \rightarrow \GL(\cL(k)).\]
Since the map $\GL(\cL(1)^{\otimes k},W) \rightarrow \GL(\cL(k))$ is Zariski-continuous, and we proved above that the image of $\Gamma$ in
$\GL(\cL(1)^{\otimes k},W)$ is irreducible, we conclude that the image of $\Gamma$ in $\GL(\cL(k))$ is irreducible with respect to the Zariski topology, as desired.
\end{proof}

\section{Finite generation for central series of \texorpdfstring{$[n]$}{[n]}-groups}
\label{section:maintools}

This section contains our main tools for proving that terms of a central series of an $[n]$-group are finitely generated.
It is divided into three subsections. In \S\ref{section:structuretheorem}, we state and prove the main technical theorem of this paper. It isolates and unifies the technical structure of the proofs of our main theorems.
The bounds in this theorem depend on two technical notions that we introduce here: the commuting graph of an $[n]$-group, 
and a new notion of ``degree of generation'' for an $[n]$-vector space endowed with a group action. In \S\ref{section:commutinggraph}, we show how to guarantee that the commuting graph is connected.
 Finally, in \S\ref{section:howtoverify} we show how to bound the degree of generation of a central series of an $[n]$-group.

\subsection{The  structure theorem}
\label{section:structuretheorem}

In this subsection, we prove the main technical theorem of this paper (Theorem~\ref{theorem:heart2} below) and its immediate consequence Corollary~\ref{corollary:heart}, which allows us to prove that certain terms of a central series of an $[n]$-group are finitely generated. 
The proof of our theorem is inspired by the proof in \S \ref{section:Kg} that $[\Torelli_g^1,\Torelli_g^1]$ is finitely
generated for $g \geq 4$. 

\para{Small subspaces}
The key objects underlying our proof that $[\Torelli_g^1,\Torelli_g^1]$ is finitely
generated were genus-$1$ BP maps.  These generate $\Torelli_g^1$ and
have two useful features: first, they
are supported on a small part of the surface, and second, they are all
conjugate under the action of $\Mod_g^1$.  For an $[n]$-vector space $V$ acted
upon by a group $\Gamma$, we will similarly want to regard elements
of the $\Gamma$-orbit of $V_I$ with $\abs{I}=d$ for a fixed small value of $d$
as being small in some sense.
This leads to the following definition.

\begin{definition}
\label{definition:gammageneration}
Let $V$ be an $[n]$-vector space and let $\Gamma$ be a group acting on $V$.  Define
\[d_{\Gamma}(V) = \min\Set{$d \geq 0$}{$V$ is generated by the $\Gamma$-orbits of its subspaces $V_I$ with $\abs{I}=d$}.\]
\end{definition}

\para{Commuting graph}
Another key feature of genus-$1$ BP maps used in \S3 is that they form the vertices
of a connected graph $\BP{g}$ whose edges correspond to commuting genus-$1$
BP maps.  The following graph will play a similar role in this section. 

\begin{definition}
\label{def:commutinggraph}
Let $\Gamma$ be an $[n]$-group.  Fix some $m \leq n$.  The {\em $m$-commuting graph} of $\Gamma$, denoted $X_{m}(\Gamma)$, is the following graph.
\begin{compactitem}
\item The vertices of $X_m(\Gamma)$ are the $\Gamma$-conjugates of the subgroups $\Gamma_I$ with $\abs{I} = m$.
\item Two vertices are joined by an edge if the associated subgroups commute elementwise.
\end{compactitem}
We say that $X_m(\Gamma)$ is \emph{nontrivial} if it consists of more than one vertex.
\end{definition}

\para{Main theorem}
Our main technical theorem is then as follows.  It will be proven at the end of this section.
 
\begin{theorem}
\label{theorem:heart2}
Let $\Gamma$ be an $[n]$-group and let $H$ and $K$ be normal subgroups of $\Gamma$ such that 
$K\subseteq H$ and $H/K$ is abelian. Let $V=H/K\otimes \mathbb R$, so $\Gamma$ acts by conjugation on $V$.
Assume the following conditions hold.
\begin{compactenum}[label=\arabic*.,ref=\arabic*]
\item\label{condition:transitive2} The $[n]$-group $\Gamma$ is transitive.
\item\label{condition:finitelygenerated2} The group $H$ is finitely generated.
\item\label{condition:Zariski2} The action of $\Gamma$ on $V$ is Zariski-irreducible.
\item\label{condition:Xmconnected2} For some
$m\geq d_{\Gamma}(V)$, the graph $X_m(\Gamma)$ is connected and nontrivial.
\end{compactenum}
Then the group $K$ is finitely generated.
\end{theorem}

\begin{remark}
The number $d_{\Gamma}(V)$ in Condition~\ref{condition:Xmconnected2}
is defined with respect to the $[n]$-vector space structure on $V$ given
by Definitions~\ref{definition:ngroupsubquotient}~and~\ref{definition:nvectorspace}. 
\end{remark}

\begin{remark}
While Theorem~\ref{theorem:heart2} does not formally require that $\Gamma$ be weakly commuting, some kind of commutativity assumption is obviously needed to ensure that the graph $X_m(\Gamma)$ is connected and nontrivial. In \S\ref{section:commutinggraph} below we show how to compute explicit bounds for weakly commuting $\Gamma$ guaranteeing that $X_m(\Gamma)$ is connected and nontrivial. 
\end{remark}

\para{Application to central series}
Before we prove Theorem \ref{theorem:heart2}, we derive the following corollary from it.  This
corollary is what we will use to prove our main results.

\begin{corollary}
\label{corollary:heart}
Let $\Gamma$ be an $[n]$-group and let $G$ be a normal subgroup of $\Gamma$.  
Let $G = G(1) \supseteq G(2) \supseteq  \cdots$ be a $\Gamma$-normal central series of $G$,
so $\Gamma$ acts by conjugation on the associated graded real Lie algebra $\cL$.  Fix $N\geq 1$, and assume the following
conditions hold.
\begin{compactenum}[label=\arabic*.,ref=\arabic*]
\item\label{condition:transitive} The $[n]$-group $\Gamma$ is transitive.
\item\label{condition:finitelygenerated} The group $G$ is finitely generated.
\item\label{condition:Zariski} The action of $\Gamma$ on $\cL(k)$ is Zariski-irreducible for all $k\geq 1$.
\item\label{condition:Xmconnected} For some
$m\geq  \max\Set{$d_{\Gamma}(\cL(k))$}{$1 \leq k < N$}$, the graph
$X_m(\Gamma)$ is connected and nontrivial.
\end{compactenum}
Then the group $G(k)$ is finitely generated for $1 \leq k \leq N$.
\end{corollary}
\begin{proof}
Apply Theorem \ref{theorem:heart2} a total of $N-1$ times, first with $H=G=G(1)$ and $K=G(2)$, then with $H=G(2)$ and $K=G(3)$, etc.
\end{proof}

\begin{remark}
In \S\ref{section:howtoverify} below we show how to effectively bound the numbers $d_{\Gamma}(\cL(k))$. 
\end{remark}

\begin{remark}
\label{remark:Lgen1}
When $\cL$ is generated by $\cL(1)$, Lemma~\ref{lemma:zariskilie} shows that to verify Condition~\ref{condition:Zariski} it suffices to check that the action of $\Gamma$ on $\cL(1)$ is Zariski-irreducible. In particular, this applies to the lower central series $G(k)=\gamma_k G$ since the graded Lie algebra of the lower central series of any group is always generated in degree 1.
\end{remark}

\para{A key lemma}
Before proving Theorem~\ref{theorem:heart2}, we establish the following key lemma which is analogous to 
\eqref{eqn:lambdakey} from the proof of Theorem~\ref{maintheorem:johnsonkernel}. The set $\Lambda$ in the lemma below
does not correspond exactly to the set $\Lambda$ appearing in \eqref{eqn:lambdakey} -- there $\Lambda$ was a generating
set for the group whose BNS invariant we were trying to understand, while in the lemma below it is a set of things
that conjugate the group whose BNS invariant we are trying to understand.
However, despite these differences the function of $\Lambda$ in our proof
is similar to that of $\Lambda$ in the proof of Theorem~\ref{maintheorem:johnsonkernel}.

\begin{lemma}
\label{lemma:key} Let $\Gamma$ be a transitive $[n]$-group and let $m \geq 1$ be such that
$X_m(\Gamma)$ is connected and nontrivial.
Let $H$ be a finitely generated normal subgroup of $\Gamma$.
Then there exists a finite subset $\Lambda \subseteq \Gamma$ with the following property.  
Let $\rho\colon H \rightarrow \R$ be a homomorphism such that
for all $\lambda \in \Lambda$, there exists some $g \in (H_{[m]})^{\lambda}=\lambda^{-1}H_{[m]}\lambda$ with $\rho(g) \neq 0$.  Then $[\rho] \in \Sigma(H)$.
\end{lemma}
\begin{proof}
Let $T$ be a finite generating set for $H$ with $1 \notin T$.  
Since $X_m(\Gamma)$ is
connected and nontrivial, the set $\Set{$(\Gamma_{[m]})^t$}{$t \in T \cup \{1\}$}$ of vertices of $X_m(\Gamma)$ must be contained in a finite nontrivial
connected subgraph $L$.  Let $\Lambda \subseteq \Gamma$ be a set containing $T \cup \{1\}$ such that the vertices of $L$ are 
$\Set{$(\Gamma_{[m]})^{\lambda}$}{$\lambda \in \Lambda$}$.  We remark that since
we insisted that $T \cup \{1\} \subseteq \Lambda$, it might be the case that
$(\Gamma_{[m]})^{\lambda} = (\Gamma_{[m]})^{\lambda'}$ for distinct
$\lambda,\lambda' \in \Lambda$.
We will prove that this set $\Lambda$ has the desired property.

Consider some $\rho\colon H \rightarrow \R$ such that
for all $\lambda \in \Lambda$, there exists some $g \in (H_{[m]})^{\lambda}$ with $\rho(g) \neq 0$.
We must prove that $[\rho] \in \Sigma(H)$.  To do this, we will use the criterion in Lemma~\ref{lemma:BNS-EH}
applied to $G=H$. 
This requires producing an appropriate sequence of elements that generate $H$, which we will do in
several steps.
We begin by enumerating $\Lambda$ as $\Lambda = \{\lambda_1,\ldots,\lambda_\ell\}$, where
the ordering is chosen such that the following hold:
\begin{compactitem}
\item $\lambda_1 = 1$.
\item For all $1 < i \leq \ell$, there exists some $1 \leq j < i$ such that
the vertices $(\Gamma_{[m]})^{\lambda_i}$ and $(\Gamma_{[m]})^{\lambda_j}$ of $L$
are distinct and joined by an edge.
\end{compactitem}
We remark that the the second condition is possible since $L$ is connected and
nontrivial (and might not be possible if $L$ were trivial -- this is where that
condition is used).  Since adjacent vertices of $L$
correspond to commuting subgroups of $\Gamma$, the following key condition holds:
\begin{equation}
\label{eqn:gammacommute}
\text{for all $2 \leq i \leq \ell$, there exists some $1 \leq j \leq i-1$ with $[(\Gamma_{[m]})^{\lambda_i}, (\Gamma_{[m]})^{\lambda_j}] = 1$}.
\end{equation}

For $1 \leq i \leq \ell$, pick elements $g_i \in (H_{[m]})^{\lambda_i}$ in
the following way.  Recall that $\lambda_1 = 1$ and that $\Lambda$ contains the generating set
$T$ for $H$; we will need to pick $g_i$ slightly more carefully when $\lambda_i\in T$.
\begin{compactitem}
\item If $\lambda_i \notin T$, then use our assumption that $\rho$ does
not vanish on $(H_{[m]})^{\lambda_i}$ to pick some $g_i \in (H_{[m]})^{\lambda_i}$
with $\rho(g_i) \neq 0$.
\item If $\lambda_i \in T$, we must be more specific; in this case set $g_i = (g_1)^{\lambda_i}$.  Since $\lambda_i \in T \subseteq H$,
we still have
\[\rho(g_i) = \rho(\lambda_i^{-1} g_1 \lambda_i) = \rho(g_1)+\rho(\lambda_i)-\rho(\lambda_i)=\rho(g_1)\neq 0.\]
\end{compactitem}
Finally, let $g_{\ell+1},\ldots,g_{r}$ be an arbitrary enumeration of $T$. We emphasize for clarity that each element $t\in T$ entails \emph{two} elements of this sequence: $(g_1)^t$ will appear among the first $\ell$ elements, and $t$ itself will appear among the last $r-\ell$ elements.

We claim that the sequence $g_1,\ldots,g_{r}$ of elements of $H$ 
satisfies the three conditions of Lemma~\ref{lemma:BNS-EH}.  
We verify these three conditions as follows.
\begin{compactitem}
\item The first says that the $g_i$ generate $H$, which is true since they contain
all the elements in the generating set $T$.
\item The second says that $\rho(g_1) \neq 0$, which is true by
construction.
\item The third says that for all $2\leq i\leq r$, there exists some 
$j<i$ such that $\rho(g_j) \neq 0$ and such that $[g_j,g_i]$ lies in
the subgroup generated by $g_1,\ldots,g_{i-1}$.
There are two cases.  The first case is where $2\leq i \leq \ell$.  As we noted above (see
\eqref{eqn:gammacommute}, and recall that $H_{[m]} \subseteq \Gamma_{[m]}$), there exists some $1 \leq j < i$ such that $[g_j,g_i] = 1$.  
Since $\rho(g_j) \neq 0$ by construction, the condition follows.  The
second case is where $\ell+1 \leq i \leq r$.  Here $g_i \in T$.  We claim in this case
that $j=1$ works.  Indeed, since $T \subseteq \Lambda$,
we have $g_i = \lambda_k$ for some $1 \leq k \leq \ell$.
By construction, we have $g_k = (g_1)^{\lambda_k}=(g_1)^{g_i}$.  Therefore 
\[[g_1,g_i] = g_1^{-1} (g_{1})^{g_i}=g_1^{-1} g_{k} \in \langle g_1,\ldots,g_{\ell}\rangle \subseteq \langle g_1,\ldots,g_{i-1}\rangle,\]
as desired.
\end{compactitem}
Lemma~\ref{lemma:BNS-EH} now implies that $[\rho]\in \Sigma(H)$.
\end{proof}

\para{Putting it all together}
We finally prove Theorem \ref{theorem:heart2}, whose statement we recall for the reader's convenience.

\def\thetheoremrepeat{\ref{theorem:heart2}}
\begin{theoremrepeat}
Let $\Gamma$ be an $[n]$-group and let $H$ and $K$ be normal subgroups of $\Gamma$ such that
$K\subseteq H$ and $H/K$ is abelian. Let $V=H/K\otimes \mathbb R$, so $\Gamma$ acts by conjugation on $V$.
Assume the following conditions hold.
\begin{compactenum}[label=\arabic*.,ref=\arabic*]
\item The $[n]$-group $\Gamma$ is transitive.
\item The group $H$ is finitely generated.
\item The action of $\Gamma$ on $V$ is Zariski-irreducible.
\item For some
$m\geq d_{\Gamma}(V)$, the graph $X_m(\Gamma)$ is connected and nontrivial.
\end{compactenum}
Then the group $K$ is finitely generated.
\end{theoremrepeat}
\begin{proof}
By Theorem~\ref{theorem:BNS}, proving that $K$ is finitely generated is equivalent to showing that
$S(H/K) \subseteq \Sigma(H)$.  Here and throughout the proof we 
identify $S(H/K)$ with the set of equivalence classes of nonzero
$\rho\colon H \rightarrow \R$ that vanish on $K$.  Our proof of this
will follow the same outline as that of Theorem~\ref{maintheorem:johnsonkernel} in \S \ref{section:Kg}, though
the fine details will be different.

Since $H$ is finitely generated and $X_m(\Gamma)$ is connected and nontrivial, we can apply Lemma \ref{lemma:key}.  Let
$\Lambda \subseteq \Gamma$ be the resulting set, so for $\rho\colon H \rightarrow \R$ the
following holds:
\begin{equation}
\label{eqn:lambdaproperty}
\text{if for all $\lambda \in \Lambda$, there exists $g \in (H_{[m]})^{\lambda}$ with $\rho(g) \neq 0$, then $[\rho] \in \Sigma(H)$}.
\end{equation}
Now consider an arbitrary nonzero $\rho \colon H \to \R$ that vanishes on $K$.  For each $\lambda \in \Lambda$, define
\[\cZ_{\lambda}=\Set{$\gamma\in \Gamma$}{$\gamma \cdot \rho$ vanishes on $(H_{[m]})^{\lambda}$}.\]
Since $\rho$ vanishes on $K$ and $V = (H/K) \otimes \R$, the map $\rho$ factors through a unique homomorphism $\orho\colon V \rightarrow \R$.  The
condition in the definition of $\cZ_{\lambda}$ is equivalent to saying that $\gamma \cdot \orho$ vanishes on $(V_{[m]})^{\lambda}$.
From this, we see that each $\cZ_{\lambda}$ is a closed subspace of $\Gamma$ in the $V$-Zariski topology.

Moreover, we claim that each $\cZ_{\lambda}$ is a proper subset of $\Gamma$.  To see this, observe first that
since $\cZ_{\lambda}$ is the translate by $\lambda \in \Gamma$ of $\cZ_1$, it suffices to check that 
$\cZ_1$ is a proper subset of $\Gamma$.  
But to have $\cZ_1=\Gamma$ would mean that $\gamma\cdot \orho$ vanishes on $V_{[m]}$ for all
$\gamma \in \Gamma$, 
or equivalently that $\orho$ vanishes on the $\Gamma$-orbit of $V_{[m]}$. Since $d_\Gamma(V)\leq m$ and 
$\Gamma$ is a transitive $[n]$-group,
the vector space $V$ is spanned by the $\Gamma$-orbit of $V_{[m]}$. 
It follows that $\orho = 0$.  This contradicts the fact that $\rho$ is nonzero, so we deduce that
$\cZ_{\lambda}$ is a proper subset of $\Gamma$, as claimed.

Recall now that $\Gamma$ is irreducible in the $V$-Zariski topology.
Since an irreducible space cannot be written as a finite union of closed proper subspaces,
we deduce that
\[\bigcup_{\lambda \in \Lambda} \cZ_{\lambda} \subsetneq \Gamma.\]
Choose some $\gamma \in \Gamma$ such that $\gamma \notin \cZ_{\lambda}$ for all $\lambda \in \Lambda$.  By definition, this means
that the restriction of $\gamma \cdot \rho$ to $(H_{[m]})^{\lambda}$ is nonzero for all $\lambda\in \Lambda$.  
We can therefore apply \eqref{eqn:lambdaproperty} 
to $\gamma\cdot \rho$ to deduce that $[\gamma\cdot \rho]\in \Sigma(H)$.
Since $\Sigma(H)$ is invariant under automorphisms of $H$, it follows that
$[\rho]\in \Sigma(H)$ as well. Since $\rho\colon H \rightarrow \R$ was an arbitrary homomorphism
vanishing on $K$, this shows that $\Sigma(H)$ contains all of $S(H/K)$, as desired.
\end{proof}

\subsection{Connectivity of the commuting graph}
\label{section:commutinggraph}

In this section we give an easy-to-verify sufficient condition for $X_m(\Gamma)$ to be connected.  

\begin{remark}
\label{remark:nontrivial}
Our results need $X_m(\Gamma)$ to be not only connected, but also nontrivial.  However,
nontriviality is a technicality that is in practice trivial to verify -- it is
enough for $\Gamma_{[m]}$ to not be a normal subgroup of $\Gamma$, which holds
for all the $\Gamma$ considered in this paper.
\end{remark}

In order to state our condition, we need some
additional terminology.

\begin{definition} 
\label{definition:complexity}
Let $\Gamma$ be an $[n]$-group and let $g \in \Gamma$.
\begin{compactitem}
\item[(a)] The \emph{complexity} of $g$,
denoted $\comp(g)$, is the smallest $c \geq 0$ such that $g\in \Gamma_I$ for some $I\subseteq [n]$ with $\abs{I}=c$.
\item[(b)] The element $g$ is \emph{good} if for any $I,J\subseteq[n]$ such that $g\in \Gamma_I$ and such that $J$ is disjoint from $I$, the element $g$ commutes 
with all elements of $\Gamma_J$.
\end{compactitem}
\end{definition}

\begin{remark}
An $[n]$-group $\Gamma$ is generated in degree $c$ if and only if it is generated by elements of complexity at most $c$, and $\Gamma$
is commuting if and only if all of its elements are good.
\end{remark}

\begin{remark}
It is reasonable to consider making a different definition, which would define the complexity of $g \in \Gamma$ instead to be the smallest $c$ such that $g$ lies in a \emph{conjugate} of $\Gamma_I$ for some
$I \subseteq [n]$ with $\abs{I}=c$. However, the above definition works better in our proofs.
\end{remark}

Our result is then as follows.  Recall that the conditions of being transitive and weakly commuting
were defined in Definition \ref{definition:threeproperties}.

\begin{proposition} 
\label{prop:connectivity}
Let $\Gamma$ be a transitive weakly commuting $[n]$-group and let $S$ be a generating set for $\Gamma$.  
Set $c=\max\Set{$\comp(s)$}{$s \in S$}$, and for some $m \geq 1$ assume that either of the following
two conditions hold:
\begin{compactenum}[label=(\alph*)]
\item $2m+c \leq n$, or
\item $2m+c-1 \leq n$ and every element of $S$ is good.
\end{compactenum}
Then $X_m(\Gamma)$ is connected.
\end{proposition}
\begin{proof} 
The proof follows the same outline as the proof of Proposition~\ref{prop:BPgraph}, though there are some minor differences. 
Without loss of generality, we can assume that $S$ is symmetric, i.e.\ that for all $s \in S$, we also have $s^{-1} \in S$.  
We must prove that there is a path in $X_m(\Gamma)$ between $\Gamma_{[m]}$ and
any other vertex.  Since $\Gamma$ is transitive, it acts transitively on the vertices of $X_m(\Gamma)$.  It
is thus enough to prove that for all $g \in \Gamma$, there is a path in $X_m(\Gamma)$ from
$\Gamma_{[m]}$ to $(\Gamma_{[m]})^{g}$.

We begin with a special case of this.  Consider some $s \in S$.  We claim that there exists
a path $\eta_s$ in $X_m(\Gamma)$ from
$\Gamma_{[m]}$ to $(\Gamma_{[m]})^{s}$.  Since $\comp(s) \leq c$ by assumption, we can 
pick $I \subseteq [n]$ with $\abs{I}= c$ such that $s \in \Gamma_{I}$.  
Set $J = [m] \cup I$.  The subgroups $\Gamma_{[m]}$ and $(\Gamma_{[m]})^s$ of $\Gamma$ both lie in $\Gamma_J$ and
$\abs{J} \leq m+c$.  We divide the proof into two cases corresponding to the two possible hypotheses
in the proposition.
\begin{compactitem}
\item The first is where $2m+c \leq n$.  We 
can then find some $K \subseteq [n]$ with $\abs{K}=m$ that is disjoint from $J$.  Since
$\Gamma$ is weakly commuting, there exists some $f \in \Gamma$ such that $(\Gamma_K)^f$ commutes with $\Gamma_J$.  Since both $\Gamma_{[m]}$ and $(\Gamma_{[m]})^s$ are contained in $\Gamma_J$, this implies that $(\Gamma_K)^f$ commutes with
both $\Gamma_{[m]}$ and $(\Gamma_{[m]})^s$.  The vertex $(\Gamma_K)^f$ of $X_m(\Gamma)$ is thus connected by an edge to both $\Gamma_{[m]}$ and $(\Gamma_{[m]})^s$, and we have found a length $2$ path in
$X_m(\Gamma)$ from $\Gamma_{[m]}$ to $(\Gamma_{[m]})^s$, as claimed.
\item The second is where $S$ consists of good elements and $2m+c-1 \leq n$.  If the intersection $[m] \cap I$ is non-empty, then $\abs{J} \leq m+c-1$ and
the argument in the previous paragraph applies.  If instead $[m]$ and $I$ are disjoint, then the fact that $s$ is good implies that
$(\Gamma_{[m]})^s = \Gamma_{[m]}$, and there is nothing to prove.
\end{compactitem}

We now deal with the general case.  Consider $g \in \Gamma$.  Since $S$ is symmetric, we can write
\[g = s_1 s_2 \cdots s_{\ell} \quad \quad \text{with $s_i \in S$}.\]
For $h \in \Gamma$ and $s \in S$, the path $(\eta_s)^h$ goes from $(\Gamma_{[m]})^h$ to $(\Gamma_{[m]})^{sh}$.
Letting $\bullet$ be the concatenation product on paths, the desired path
from $\Gamma_{[m]}$ to $(\Gamma_{[m]})^{g}$ is then
\[\eta_{s_{\ell}} \bullet (\eta_{s_{\ell-1}})^{s_{\ell}} \bullet (\eta_{s_{\ell-2}})^{s_{\ell-1} s_{\ell}} \bullet \cdots \bullet (\eta_{s_{1}})^{s_2 s_3 \cdots s_{\ell}}.\qedhere\]
\end{proof}

\subsection{Bounding \texorpdfstring{$d_\Gamma(\cL(k))$}{d-Gamma(L(k))}}
\label{section:howtoverify} 

In this section, we show how to effectively bound the numbers $d_{\Gamma}(\cL(k))$ 
in the statement of Corollary~\ref{corollary:heart}.
We start with the following lemma (see Definition~\ref{definition:complexity} for the definition
of $\comp(v)$).  

\begin{lemma}
\label{lemma:boundcommutator}
Let $G$ be an $[n]$-group and let
$G=G(1)\supseteq G(2)\supseteq \cdots$ be a central series.  Let
$\cL = \bigoplus \cL(k)$ be the associated graded real Lie algebra, and endow each $\cL(k)$ with the $[n]$-vector space structure
induced by the $[n]$-group structure on $G$ (see Definition~\ref{def:[n]vectorspacestructure}).
Given $v \in \cL(k)$ and $v' \in \cL(k')$, consider $[v,v']\in \cL(k+k')$.  The following then hold.
\begin{compactenum}[label=(\roman*)]
\item\label{part:compgeneric} We have $\comp([v,v']) \leq \comp(v) + \comp(v')$.
\item\label{part:compweakly} If $G$ is weakly commuting, then $\comp([v,v']) \leq \max\{\comp(v)+\comp(v')-1,0\}$.
\end{compactenum}
\end{lemma}
\begin{proof}
Given $g\in G(m)$, let
\[\pi_m(g) \in \cL(m) = (G(m)/G(m+1)) \otimes \R\]
denote its projection to $\cL(m)$, that is, $\pi_m(g)=gG(m+1)\otimes 1$. Note that
\[\pi_{m+m'}([g,g'])=[\pi_m(g),\pi_{m'}(g')]\] for all $g\in G(m)$ and $g'\in G(m')$, by the definition of the bracket on $\cL$.

Choose $I \subseteq [n]$ such that $v\in \cL(k)_{I}$ and $\abs{I}=\comp(v)$,
and similarly choose $I' \subseteq [n]$ such that $v' \in \cL(k')_{I'}$ and $\abs{I'}=\comp(v')$.  By the definition of the $[n]$-vector space
structures on the $\cL(m)$, we can write
\[v = \sum_{i=1}^{r} \lambda_i \pi_k(g_i) \quad \quad \text{and} \quad \quad v' = \sum_{j=1}^{r'} \lambda'_j \pi_{k'}(g_j'),\]
where $\lambda_i,\lambda'_j \in \R$ and $g_i \in G(k)_I$ and $g_j' \in G(k')_{I'}$.  We then have
\[ [v,v'] = \sum_{i,j} \lambda_i \lambda'_j [\pi_k(g_i), \pi_{k'}(g_j')] = \sum_{i,j} \lambda_i \lambda'_j \pi_{k+k'}([g_i,g_j']).\]
Since $[g_i,g_j']\in G(k+k')\cap G_{I\cup I'}=G(k+k')_{I\cup I'}$, it follows that
$\pi_{k+k'}([g_i,g_j'])\in \cL(k+k')_{I\cup I'}$.  We deduce that $[v,v']\in \cL(k+k')_{I\cup I'}$.
Since $\abs{I \cup I'}\leq \abs{I}+\abs{I'}=\comp(v)+\comp(v')$, this proves~\ref{part:compgeneric}.

Suppose now that $G$ is weakly commuting.
If $I \cap I' \neq\emptyset$, then $\abs{I \cup I'}\leq \abs{I}+\abs{I'}-1$, and \ref{part:compweakly} follows.
If instead $I\cap I'=\emptyset$, then there exists $x\in G$ such that $(G_{I})^x$ and $G_{I'}$ commute. Since
\[g_i^{-1}(g_i)^x=[g_i,x]\in G(k+1),\]
we have $\pi_k(g_i)=\pi_k((g_i)^x)\in \cL(k)$, so
\[\pi_{k+k'}([g_i,g_j'])=\pi_{k+k'}([(g_i)^x,g_j'])\in \cL(k+k').\]
But since $(g_i)^x$ and $g_j'$ commute for all $i$ and $j$, we have
$\pi_{k+k'}([(g_i)^x,g_j'])=\pi_{k+k'}(1)=0$. It follows that $[v,v']=0$, and in particular that
$\comp([v,v'])=0$, proving~\ref{part:compweakly}.
\end{proof}

We can now prove our main proposition.  Recall that the quantity $d_{\Gamma}(V)$
is defined for an arbitrary $[n]$-vector space $V$ endowed with an action of a group $\Gamma$ (see Definition~\ref{definition:gammageneration}).

\begin{proposition} 
\label{prop:generalbound}
Let $\Gamma$ be an $[n]$-group, let $G$ be a normal subgroup of $\Gamma$, and let 
$G=G(1)\supseteq G(2)\supseteq \cdots$ be a $\Gamma$-normal central series.  Let
$\cL = \bigoplus \cL(k)$ be the associated graded real Lie algebra, and assume that
$\cL$ is generated by $\cL(1)$.  Set $d=d(\cL(1))$.
If $G$ is weakly commuting, set $e=d-1$; otherwise, set $e=d$.  Then for all $k\geq 2$, we have
$d_\Gamma(\cL(k)) \leq d_\Gamma(\cL(k-1)) + e$.  In particular, by induction we have 
\[d_\Gamma(\cL(k)) \leq d_\Gamma(\cL(1)) +(k-1)e\leq d+(k-1)e.\]
\end{proposition}

\begin{remark}
We emphasize that the definition of $e$ in Proposition~\ref{prop:generalbound} depends on whether or not the \emph{normal} subgroup $G$
is weakly commuting as an $[n]$-group.  This is a stronger condition than the hypothesis in Proposition~\ref{prop:connectivity},
which only asserts that $\Gamma$ is weakly commuting.
\end{remark}

\begin{proof}[Proof of Proposition~\ref{prop:generalbound}]
Fix $k\geq 2$. Our goal is to show that 
\[d_\Gamma(\cL(k))\leq d_\Gamma(\cL(k-1))+e.\] 
In other words, we must show that any $v\in \cL(k)$ can be written as a finite sum
of {$\Gamma$\nobreakdash-conjugates} of elements of complexity at most $d_\Gamma(\cL(k-1))+e$.
Since the Lie algebra $\cL$ is generated by $\cL(1)$, we can write $v$ as a finite sum of elements of the form
$[w,s]$ with $w \in \cL(k-1)$ and $s \in \cL(1)$.  Since the desired conclusion is closed under addition, it
suffices to handle the case of a single term, i.e.\ the case where $v = [w,s]$ with $w \in \cL(k-1)$ and $s \in \cL(1)$.

By the definition of $d_\Gamma(\cL(k-1))$, we can write $w$ as a finite sum of elements of the form
$u^{\gamma}$ with $\gamma \in \Gamma$ and $u \in \cL(k-1)$ satisfying $\comp(u) \leq d_\Gamma(\cL(k-1))$.
Since the Lie bracket on $\cL$ is bilinear and our desired conclusion is closed under addition, it again
suffices to handle the case of a single term, i.e.\ the case where 
$v=[u^{\gamma},s]$ with $\gamma \in \Gamma$ and $u \in \cL(k-1)$ satisfying $\comp(u) \leq d_\Gamma(\cL(k-1))$.

Since $\Gamma$ acts on $\cL$ by Lie algebra automorphisms, we have 
$[u^{\gamma},s] = [u,s^{\gamma^{-1}}]^{\gamma}$.  By the definition of $d(\cL(1))$, we can write
\[s^{\gamma^{-1}} = \sum_{i=1}^r s_i \quad \quad \text{with $s_i \in \cL(1)$ satisfying $\comp(s_i) \leq d$}.\]
It follows that
\begin{equation}
\label{eq:v-ugamma-s}
v=[u^{\gamma},s] = [u,s^{\gamma^{-1}}]^{\gamma} = \sum_{i=1}^r [u,s_i]^{\gamma}.
\end{equation}
Letting $\varepsilon=1$ if $G$ is weakly commuting and $\varepsilon=0$ otherwise, Lemma \ref{lemma:boundcommutator} implies that for each $i$ we have
\[\comp([u,s_i]) \leq \comp(u) + \comp(s_i) - \varepsilon \leq d_\Gamma(\cL(k-1)) + d - \varepsilon = d_\Gamma(\cL(k-1)) + e,\]
as desired.
\end{proof}

\section{Proofs of Theorems~\ref{maintheorem:torellilcs},~\ref{maintheorem:ialcs}, and \ref{maintheorem:iajohnson}}
\label{section:mainproofs}

In this section, we prove Theorems~\ref{maintheorem:torellilcs},~\ref{maintheorem:ialcs}, and \ref{maintheorem:iajohnson}.  The
bounds in our theorems are stronger than what can be obtained from a completely general framework, so we
will need to use some rather special properties of the groups in question.  There are three sections.
First, in \S\ref{section:torellilcs} 
we prove Theorem~\ref{maintheorem:torellilcs} on the lower central series of $\Torelli_g^1$.
Theorem~\ref{maintheorem:torellijohnson} on the Johnson filtration of $\Torelli_g^1$ is a special case of
Theorem~\ref{maintheorem:torellilcs}, so there is no need to prove it separately.  Next, in  
\S\ref{section:ialcs} we prove Theorem~\ref{maintheorem:ialcs} on the lower central series
of $\IA_n$. 
Finally, in \S\ref{section:iajohnson} we prove Theorem~\ref{maintheorem:iajohnson} on the Johnson filtration of $\IA_n$. 

\begin{remark}
The proof of Theorem~\ref{maintheorem:ialcs} in \S\ref{section:ialcs} has the fewest technicalities, so we suggest
reading it first. It can be understood independently without first reading \S\ref{section:torellilcs}.
\end{remark}

\subsection{The lower central series of \texorpdfstring{$\Torelli^1_g$}{I-g,1}}
\label{section:torellilcs}

The goal in this subsection is to prove Theorem~\ref{maintheorem:torellilcs} concerning the lower central series
of $\Torelli_g^1$.

\para{Notation}
The following notation will be in place for the remainder of this subsection.  Fix some $g \geq 3$.
Let $\Gamma = \Mod_g^1$, let $G = \Torelli_g^1$, and let $G(k) = \gamma_k \Torelli_g^1$.  Finally, let
$\cL = \bigoplus \cL(k)$ be the graded real Lie algebra associated to $G(k)$.  Endow $\Gamma$ with the
$[g]$-group structure described in Definition~\ref{definition:ngroupmodg}. Recall that the subgroups 
$G$ and $G(k)$ inherit a $[g]$-group structure,  and the vector spaces $\cL(k)$ inherit a $[g]$-vector space structure.

\para{Basic properties}
Our goal is to apply Corollary~\ref{corollary:heart} to the filtration $G(k)$ of $\Gamma$.  That
corollary has several conditions.  The following lemma verifies the first of them.
Recall that we defined what it means for a $[g]$-group structure to be
transitive and weakly commuting in Definition~\ref{definition:threeproperties}.

\begin{lemma}
\label{lemma:torellibasic}
The $[g]$-group $\Gamma$ is transitive and the $[g]$-groups $\Gamma$ and $G$ are weakly commuting, but not commuting.
\end{lemma}
\begin{proof}
That $\Gamma$ is transitive is a direct consequence of \cite[Lemma~4.1(i)]{CP}, which
says that for all $1 \leq k \leq g$ the group $\Mod_g^1$ acts transitively 
on isotopy classes of subsurfaces of $\Sigma_g^1$ which are homeomorphic to $\Sigma_k^1$.
We thus must only prove that $\Gamma$ and $G$ are weakly commuting.
Let $I,J\subseteq [g]$ be disjoint.  Recall that in Definition~\ref{definition:ngroupmodg}, we defined
$\Gamma_I$ and $\Gamma_J$ to consist of mapping classes supported on the genus $|I|$ and $|J|$ subsurfaces $\Sigma_I$ and
$\Sigma_J$ constructed by Lemma \ref{lemma:choosesurfaces} and illustrated in Figure \ref{figure:bigsurface}.
As discussed in conclusion \ref{part:disjointconsecutive} of Lemma \ref{lemma:choosesurfaces}, the 
surfaces $\Sigma_I$ and $\Sigma_J$ need not be homotopic to disjoint subsurfaces, so $\Gamma_I$ and $\Gamma_J$
need not commute and thus $\Gamma$ and $G$ are not commuting.  However, we can always find a subsurface
$\Sigma'_J$ which is homeomorphic to
$\Sigma_J$, disjoint from $\Sigma_I$, and satisfies $\HH_1(\Sigma_J)=\HH_1(\Sigma'_J)$ as subspaces of $\HH_1(\Sigma)$.
By \cite[Lemma~4.2(ii)]{CP} this implies that there exists $\varphi \in G$ such that $\varphi(\Sigma'_J)=\Sigma_J$, 
so the subgroup $(\Gamma_J)^{\varphi}=\varphi^{-1}\Gamma_J \varphi$
consists of mapping classes supported on $\Sigma'_J$.  It follows that $(\Gamma_J)^{\varphi}$ commutes with $\Gamma_I$. 
Since $\varphi \in G$, this shows that both $G$ and $\Gamma$ are weakly commuting.
\end{proof}

\para{Generating $\mathbf{G(1)}$}
The second condition in Corollary~\ref{corollary:heart} is that $G(1) = \Torelli_g^1$ is finitely generated.  This
was proved by Johnson \cite{JohnsonFinite}, and stated above as Theorem~\ref{theorem:johnson}\ref{part:johnson:fg}.

\para{Zariski-irreducibility}
The third condition in Corollary~\ref{corollary:heart} is that the action of $\Gamma = \Mod_g^1$ on
each $\cL(k)$ is Zariski-irreducible, which is the content of the following.

\begin{lemma}
\label{lemma:torellizariski}
For all $k \geq 1$, the action of $\Gamma$ on $\cL(k)$ is Zariski-irreducible.
\end{lemma}
\begin{proof}
Since $G(k)=\gamma_k G$ is the lower central series of $G$, by Remark~\ref{remark:Lgen1} it
suffices to prove that the action on $\cL(1)=(\Torelli_g^1)^{\ab} \otimes \R$ is Zariski-irreducible.  But this has already been proved for the dual representation $\cL(1)^*=\Hom(\Torelli_g^1,\R)$ in Lemma \ref{lemma:zariskimodirred} which suffices according to Remark~\ref{remark:dualZariski}.
\end{proof}

\para{Connectivity bounds}
The fourth condition in Corollary~\ref{corollary:heart} asserts that the graph $X_m(\Gamma)$ must be connected and nontrivial for
some $m \geq \max\Set{$d_{\Gamma}(\cL(k))$}{$1 \leq k < N$}$. The following lemmas will allow
us to verify this.

\begin{lemma}
\label{lemma:torellixm}
For all $m \geq 1$ such that $2m+1 \leq g$, the graph $X_m(\Gamma)$ is connected and nontrivial.
\end{lemma}

\noindent
For the proof of Lemma~\ref{lemma:torellixm}, we need the following fact about generators for $\Gamma = \Mod_g^1$.
Recall that we defined what it means for an element of a $[g]$-group to be good in Definition~\ref{definition:complexity}.

\Figure{figure:redrawngenerators}{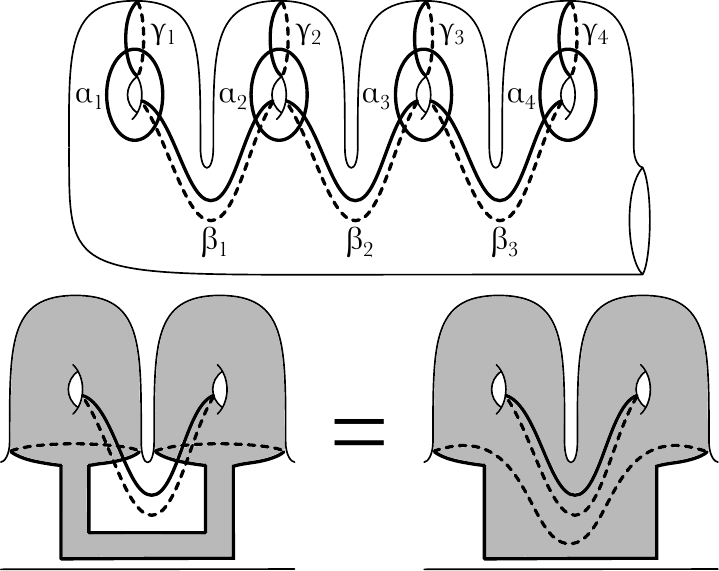}{On the top are the generators for $\Gamma = \Mod_g^1$, drawn
in such a way that it is clear how they interact with the subsurfaces $\Sigma_I$.  On the bottom
we illustrate why $\beta_j$ lies in a surface isotopic to $\Sigma_{j,j+1}$, and thus why the Dehn twist
$T_{\beta_j}$ lies in $\Gamma_{j,j+1}$.}{78}

\begin{lemma}
\label{lemma:modg1gen}
The $[g]$-group $\Gamma$ is generated by good elements of complexity at most $2$.
\end{lemma}
\begin{proof}
The key point here is that $\Gamma$ is generated by the union of the subgroups $\Gamma_{\{i,i+1\}}$ for
$1 \leq i \leq g-1$.  To see this, observe that in Figure \ref{figure:redrawngenerators} we have redrawn
Johnson's generating set for $\Gamma = \Mod_g^1$ from \cite[Theorem 1]{JohnsonFinite} (we previously
used this generating set in the proof of Proposition~\ref{prop:BPgraph}, where it is depicted
in Figure \ref{figure:bpgraph}).  This generating set consists of Dehn twists
$T_{\alpha_i}$ and $T_{\beta_j}$ and $T_{\gamma_k}$ with $1 \leq i,k \leq g$ and $1 \leq j \leq g-1$.
Letting $\Sigma_I$ be the surfaces given by Lemma \ref{lemma:choosesurfaces}
and used to define the $[g]$-group structure on $\Gamma$ in Definition \ref{definition:ngroupmodg}, it
is clear from this picture that $T_{\alpha_i} \in \Gamma_{\{i\}}$ and 
$T_{\beta_j} \in \Gamma_{\{j,j+1\}}$ and $T_{\gamma_k} \in \Gamma_{\{k\}}$ (this is slightly
nontrivial for $T_{\beta_j}$, for which we refer the reader to the bottom of this figure).
We remark that an alternate algebraic proof that $\Gamma$ is generated by the union of the subgroups $\Gamma_{\{i,i+1\}}$
is in \cite{McCool}.  Conclusion~\ref{part:disjointconsecutive} from Lemma \ref{lemma:choosesurfaces} implies that any element of $\Gamma_{\{i,i+1\}}$ is good.  Since these elements have complexity at most $2$, this verifies the lemma.
\end{proof}

\begin{proof}[Proof of Lemma~\ref{lemma:torellixm}]
We will use Proposition~\ref{prop:connectivity}.  Let $S \subset \Mod_g^1$ be the generating set given by Lemma~\ref{lemma:modg1gen}.
Set $c=\max\Set{$\comp(s)$}{$s \in S$}$, so $c=2$.  Every element of $S$ is good.
Proposition~\ref{prop:connectivity} thus says that $X_m(\Gamma)$ is connected whenever $2m+c-1 = 2m+1 \leq g$. 
As for the nontriviality of $X_m(\Gamma)$, it follows immediately from the fact
that $\Mod_m^1$ is never normal in $\Mod_g^1$ except when $m=0$ (when $\Mod_m^1 = 1$) and when
$m=g$ (when $\Mod_m^1 = \Mod_g^1$); see Remark \ref{remark:nontrivial}.
\end{proof}

\begin{remarks}
\mbox{}
\begin{compactenum}
\item The commuting graph $X_m(\Gamma) = X_m(\Mod_g^1)$ has a more geometric description.
The conjugates of subgroups $(\Mod_g^1)_I$ with $\abs{I}=m$ are in bijection with isotopy classes of subsurfaces  of $\Sigma$ homeomorphic to $\Sigma_m^1$; such a subsurface is sometimes called an $m$-handle. Two such subgroups commute if and only if the corresponding $m$-handles are disjoint. Accordingly the graph $X_m(\Mod_g^1)$ forms the $1$-skeleton of the ``$m$-handle complex'', whose vertices are $m$-handles and whose simplices consist of disjoint $m$-handles.  For $m=1$, this first appeared in
\cite{PutmanSamLinear}, where it was proved to be $(g-3)/2$-connected (this was deduced from a similar connectivity
result for a slightly different complex by Hatcher--Vogtmann \cite{HatcherVogtmannTethers}).
The fact that the $m$-handle complex is connected for $g\geq 2m+1$ must be well known, although we are not aware of an explicit reference in the literature.  We remark that after a first version of this paper was circulated, the third
author proved that the $m$-handle complex is actually
\[\frac{g-(2m+1)}{m+1}-\text{connected};\]
see \cite[Theorem D]{PutmanPartial}.
\item We also see that the bound $g\geq 2m+1$ is sharp. For $g<2m$ there are no edges in $X_m(\Mod_g^1)$, since there cannot be two disjoint $m$-handles. For $g=2m$, an $m$-handle determines a splitting of $\HH_1(\Sigma_g^1)$ into two rank-$m$ symplectic subspaces; disjoint $m$-handles determine the same splitting, so this invariant is constant on components of $X_m(\Mod_g^1)$.
\item Finally, we remark that the  genus-1 BP graph appearing in \S\ref{section:Kg} can be thought of as the ``1.5-handle complex'', and note that we proved there that this is connected whenever $4=2(1.5)+1\leq g$, matching Lemma~\ref{lemma:torellixm}.
\end{compactenum}
\end{remarks}

\para{Generation for $\cL(k)$} Recall that a general upper bound on the quantities $d_\Gamma(\cL(k))$ was obtained in 
Proposition~\ref{prop:generalbound}. We will now use this proposition to obtain more specific bounds in the case
$\Gamma=\Mod_g^1$ and $G=\Torelli_g^1$:

\begin{proposition}
\label{prop:torellidgamma}
We have $d_{\Gamma}(\cL(1)) = 2$ and $d_\Gamma(\cL(k))\leq k$ for $k \geq 2$.
\end{proposition}
\begin{proof} Set $V_{\Z} = \HH_1(\Sigma_g^1;\Z)$ and $V = \HH_1(\Sigma_g^1;\R)$.
For each $1\leq i\leq g$, let $\{a_i,b_i\}$ be a symplectic basic for $\HH_1(\Sigma_{\{i\}};\Z)$, so
$\mathcal B=\{a_i,b_i\}_{i=1}^g$ is a symplectic basis for $V_{\Z}$.  For each $I\subseteq [n]$ let
\[V_{I,\Z}=\bigoplus_{i\in I}(\Z a_i\oplus \Z b_i)\qquad\text{ and }\qquad V_I=V_{I,\Z}\otimes \R=\bigoplus_{i\in I}(\R a_i\oplus \R b_i).\]
Theorem~\ref{theorem:johnson}\ref{part:johnson:H1} states that $\cL(1) = G^{\ab} \otimes \R \iso\bwedge^3 V$.
Moreover, it follows from Johnson's work in \cite{JohnsonAbel} that the image of $G_I$ in $G^{\ab}\otimes \R$ is equal to $\bwedge^3 V_{I,\Z}$ (even if $\abs{I}\leq 3$), so $\cL(1)_I=\bwedge^3 V_I$. In particular, this shows that $d(\cL(1))\leq 3$ since each basis element of $\bwedge^3 V$ involves three elements of $\{a_1,b_1,\ldots,a_g,b_g\}$, and thus lies in $\bwedge^3 V_I$ for some $I$ with $\abs{I}\leq 3$. It is easy to see that elements of complexity at most $2$ cannot span $\bwedge^3 V$, so in fact  $d(\cL(1))=3$. We now tackle $d_\Gamma(\cL(k))$ for different $k$ in turn.

$\mathbf{k=1}$:\quad Consider the element $a_1\wedge a_2\wedge b_2\in \cL(1)$, which has  complexity~2 since it belongs to $\cL(1)_{\{1,2\}}$. The $\Sp_{2g}(\Z)$-orbit of this element spans $\bwedge^3 V$. This can be seen either algebraically, since $\bwedge^3 V$ contains only two irreducible $\Sp_{2g}(\Z)$-representations and this element is not contained in either, or via Theorem~\ref{theorem:johnson}\ref{part:johnson:BP} (this is the image of a genus-$1$ BP map, and Theorem~\ref{theorem:johnson}\ref{part:johnson:BP} states that $\Torelli_g^1$ is generated by the $\Gamma$-orbit of such an element).
Since $\cL(1)$ is spanned by the $\Gamma$-orbit of this complexity-2 element, we conclude that $d_\Gamma(\cL(1))=2$
(we cannot have $d_\Gamma(\cL(1))\leq 1$ since $\cL(1)_I=0$ if $\abs{I}=1$).

$\mathbf{k=2}$:\quad Next, we prove that $d_\Gamma(\cL(2))\leq 2$ using a rather different argument, resting on two important results of Johnson that we have not used thus far. Johnson~\cite[Theorem~1]{JohnsonHomeo} proved that for $g\geq 3$ the Johnson kernel
$\JohnsonKer_g^1$ is generated by the set $S$ of separating twists of genus $1$ and $2$, that is,
Dehn twists about separating curves that cut off subsurfaces homeomorphic to either $\Sigma_1^1$ or $\Sigma_2^1$.
Any separating curve of genus $1$ or $2$ is in the $\Mod_g^1$-orbit of the boundary of $\Sigma_{\{1\}}$ or $\Sigma_{\{1,2\}}$ respectively.  Therefore, $\JohnsonKer_g^1$ is generated by the $\Mod_g^1$-conjugates of $\JohnsonKer_g^1\cap \Gamma_{\{1,2\}}$.
Johnson~\cite{JohnsonAbel} also proved that $\gamma_2 G$ is a finite index subgroup of $\JohnsonKer_g^1$ 
and that $\JohnsonKer_g^1/\gamma_2 G\iso (\Z/2)^r$ for some $r \geq 1$.  Let $S'=\Set{$s^2$}{$s\in S$}$
be the set of squares of separating twists of genus $1$ and $2$ and let $H\subseteq \gamma_2 G$ be the subgroup generated by $S'$. The group $H$ is normal in $\Gamma$ since $S$ (and hence $S'$) is closed under conjugation in $\Gamma$.
Note that $H$ need not have finite index in $ \JohnsonKer_g^1$, but the quotient $\JohnsonKer_g^1/H$ is generated by \emph{torsion} elements (namely the order-2 elements that are the image of $S$).

Now consider the image of $H\subseteq \gamma_2 G\subseteq \JohnsonKer_g^1$ under the natural projection $\rho\colon \JohnsonKer_g^1\to \JohnsonKer_g^1/\gamma_3 G$. Since $G/\gamma_3 G$ is finitely generated nilpotent, the same is true of its subgroup $\rho(\JohnsonKer_g^1)$. Therefore its quotient $\rho(\JohnsonKer_g^1)/\rho(H)$ is finitely generated, nilpotent, and  generated by torsion elements, and thus is finite.
This means that $\rho(H)$ is finite index in $\rho(\JohnsonKer_g^1)$, and therefore in the intermediate subgroup $\rho(\gamma_2 G)=\gamma_2 G/\gamma_3 G$. Tensoring with $\R$, we conclude that the image of $H$ spans all of $(\gamma_2 G/\gamma_3 G)\otimes\R=\cL(2)$. Since $H$ is generated by $\Gamma$-conjugates of elements of $G_{\{1,2\}}$, we conclude that $d_\Gamma(\cL(2))\leq 2$.

$\mathbf{k\geq 3}$:\quad To conclude the proof, we will modify the proof of Proposition~\ref{prop:generalbound} to show that 
\begin{equation}
\label{eq:dGamma-ineq}
d_\Gamma(\cL(k))\leq d_\Gamma(\cL(k-1))+1
\end{equation}
for $k \geq 3$; the bound $d_\Gamma(\cL(k))\leq k$ then follows by induction. Fix $k\geq 3$. Recall from above that $d(\cL(1))=3$. Since $G$ is weakly commuting, the proof of
Proposition~\ref{prop:generalbound} (specifically equation~\eqref{eq:v-ugamma-s}) shows that $\cL(k)$ is generated by the $\Gamma$-orbits of elements of the form
$[v,s]$ where $v\in \cL(k-1)_I$ and $s\in \cL(1)_J$ for some $I,J\subseteq [g]$ with $\abs{I}\leq d_\Gamma(\cL(k-1))$ and $\abs{J}\leq 3$ and $I\cap J\neq\emptyset$. We may assume that $s$ is a standard basis element of $\cL(1)\iso \bwedge^3 V$ and that $I$ and $J$ are as small as possible.

Note that 
\[\comp([v,s])\leq \abs{I\cup J}\leq d_\Gamma(\cL(k-1))+1\] 
unless $\abs{J}=3$ and $\abs{I\cap J}=1$, so assume that the latter is the case. Let
$r$ be the unique element of $I\cap J$ and $t$ and $u$ the other two elements of $J$. Since $\comp(s)=3$ we must have $s=x\wedge y\wedge z$ where $x=a_r$ or $b_r$, $y=a_{t}$ or $b_t$, and $z=a_u$ or $b_u$.
Set $w=[v,x\wedge a_t\wedge b_t]$. Since $x\wedge a_t\wedge b_t\in \cL(1)_{\{r,t\}}$ we have $w\in \cL(k)_{I\cup \{r,t\}}=\cL(k)_{I\cup \{t\}}$, so $\comp(w)\leq d_\Gamma(\cL(k-1))+1$. Using the action of $\Sp_{2g}(\Z)$ we will show that the $\Gamma$-orbit of $w$ contains
$[v,x\wedge a_t\wedge a_u]$ and $[v,x\wedge a_t\wedge b_u]$ and $[v,x\wedge b_t\wedge a_u]$ and $[v,x\wedge b_t\wedge b_u]$.
Since $[v,s]$ must be equal to one of these, this will finish the proof of \eqref{eq:dGamma-ineq}.

Consider the symplectic automorphisms $\sigma_i$ for $i\in [g]$ and $\tau_{ij}$ for $i\neq j\in [g]$ of $V_\Z$ defined as follows (all basis
elements whose image is not specified are fixed):
\[\sigma_i\colon\begin{cases}a_i\mapsto b_{i}\\
b_i\mapsto -a_{i}\end{cases}
\qquad\qquad\qquad
\tau_{ij}\colon \begin{cases}b_i\mapsto b_i+ a_j\\
b_j\mapsto b_j+a_i\end{cases}.\]
Note that
\begin{align*}
\tau_{tu}(a_t\wedge b_t)&=a_t\wedge (b_t+a_u)=
a_t\wedge b_t+a_t\wedge a_u\\
\sigma_t(a_t\wedge a_u)&=b_t\wedge a_u\\
\sigma_u(a_t\wedge a_u)&=a_t\wedge b_u\\
\sigma_t(a_t\wedge b_u)&=b_t\wedge b_u.
\end{align*}
This shows that the span of the orbit of $a_t\wedge b_t\in V\wedge V$ under the subgroup generated by $\{\sigma_t, \sigma_u,\tau_{tu}\}$ contains $a_t\wedge a_u$ and $b_t\wedge a_u$ and
$a_u\wedge b_t$ and $a_u\wedge b_u$.
By construction $\sigma_t$ and $\sigma_u$ and $\tau_{tu}$ fix $V_I$; this implies that they fix $v\in \cL(k-1)_I$, since we may lift these automorphisms to elements of $\Gamma$ that fix every element of $\Gamma_I$. They also fix $x\in V$. Therefore applying the computations above to
$w=[v,x\wedge a_t\wedge b_t]$ shows that the $\Gamma$-orbit of $w$ contains the claimed elements; for example, $\tau_{tu}(w)-w=[v,x\wedge a_t\wedge a_u]$, and so on.
\end{proof}

\para{Putting it all together}
All the pieces are now in place to prove Theorem~\ref{maintheorem:torellilcs}.

\begin{proof}[Proof of Theorem~\ref{maintheorem:torellilcs}]
The notation is as above.  As was established in \S \ref{section:outlineproof}, we must
prove that $G(k)$ is finitely generated for $k \geq 3$ and $g \geq 2k-1$, or equivalently when
$3\leq k \leq \frac{g+1}{2}$.
We will apply Corollary~\ref{corollary:heart} with $N = \lfloor\frac{g+1}{2}\rfloor$and $m=N-1=\lfloor\frac{g-1}{2}\rfloor$. 
This theorem has four hypotheses:
\begin{compactitem}
\item The $[g]$-group $\Gamma$ must be transitive, which is one of the conclusions of Lemma~\ref{lemma:torellibasic}.
\item The group $G$ must be finitely generated, which is Theorem~\ref{theorem:johnson}\ref{part:johnson:fg}.
\item The action of $\Gamma$ on each $\cL(k)$ must be Zariski-irreducible, which is
Lemma~\ref{lemma:torellizariski}.
\item The graph $X_m(\Gamma)$ must be connected and nontrivial, and we must have
\begin{equation}
\label{eqn:modmbound}
m \geq \max\Set{$d_{\Gamma}(\cL(k))$}{$1 \leq k < N$}.
\end{equation}
To see that $X_m(\Gamma)$ is connected and nontrivial, it is enough to verify the two hypotheses
of Lemma~\ref{lemma:torellixm}.  The first is that $m \geq 1$; indeed, since $g \geq 2k-1 \geq 5$, we have
\[m=\lfloor\frac{g-1}{2}\rfloor\geq 2 \geq 1.\]
The second is that $2m+1 \leq g$; indeed,
\[\textstyle 2m+1=2\big\lfloor\frac{g-1}{2}\big\rfloor+1\leq 2\cdot\frac{g-1}{2}+ 1 = g.\]
As for \eqref{eqn:modmbound},
Proposition~\ref{prop:torellidgamma} says that
$d_{\Gamma}(\cL(1)) = 2$ and that $d_{\Gamma}(\cL(k)) \leq k$ for $k \geq 2$, so since $N-1\geq 2$ we have 
\[\max\Set{$d_{\Gamma}(\cL(k))$}{$k\leq N-1$} \leq N-1=m,\]
as desired.
\end{compactitem}
Applying Corollary~\ref{corollary:heart}, we conclude that $G(k)$ is finitely generated for $1 \leq k \leq N$.
\end{proof}

\subsection{The lower central series of \texorpdfstring{$\AutFn$}{AutFn}}
\label{section:ialcs}

The goal in this section is to prove Theorem~\ref{maintheorem:ialcs} concerning the lower central series
of $\IA_n$.

\para{Setup}
Recall that in Definition~\ref{definition:ngroupautfn} we defined an $[n]$-group structure on $\AutFn$.  There is a minor 
technical problem that will prevent us from working with $\AutFn$ directly.  To explain this,
consider the map $\AutFn \rightarrow \GL_n(\Z)$ arising from the action of $\AutFn$ on $F_n^{\ab} = \Z^n$.
This map is surjective, and the Zariski closure of $\GL_n(\Z)$ in $\GL_n(\R)$ is the group $\SL^{\pm}_n(\R)$
of matrices whose determinant is $\pm 1$.  The group $\SL^{\pm}_n(\R)$ is not connected, so the pullback of
the Zariski topology on $\GL_n(\Z)$ to $\AutFn$ does {\em not} make $\AutFn$ into an irreducible space.
To correct this, we will instead work with the group $\SAutFn$ consisting of elements of $\AutFn$ that
act on $F_n^{\ab}$ with determinant $1$.  Since $\SAutFn$ is a subgroup of $\AutFn$, it inherits an
$[n]$-group structure.

\para{Notation}
The following notation will be in place for the remainder of this section.  Fix some $n \geq 2$.
Let $\Gamma = \SAutFn$, let $G = \IA_n$, and let $G(k) = \gamma_k \IA_n$.  Finally, let
$\cL = \bigoplus \cL(k)$ be the graded real Lie algebra associated to $G(k)$.  The groups $\Gamma$ and
$G$ and $G(k)$ are endowed with the $[n]$-group structure coming from the $[n]$-group structure on $\AutFn$, and
the vector spaces $\cL(k)$ is endowed with the induced $[n]$-vector space structure.

\para{Basic properties}
Our goal is to apply Corollary~\ref{corollary:heart} to the filtration $G(k)$ of $\Gamma$.  That
corollary has several conditions.  The following lemma verifies the first of them.  Recall
that the we defined what it means for an $[n]$-group to be commuting and transitive in
Definition~\ref{definition:threeproperties}.

\begin{lemma}
\label{lemma:sautfnbasic}
The $[n]$-group $\Gamma = \SAutFn$ is commuting and transitive.
\end{lemma}
\begin{proof}
We have already noted in Remark~\ref{remark:basicproperties} that the $[n]$-group $\AutFn$ is commuting, 
so the same is true of its subgroup $\SAutFn$.  To see that it is transitive, consider the
subgroup of $\AutFn$ preserving the set $\{x_1,x_1^{-1},\ldots,x_n,x_n^{-1}\}$, which we identify with the signed permutation group $S_n^\pm$.  If 
$\widetilde{\sigma}\in S_n^\pm$ projects to $\sigma\in S_n$, then from the definition of $\AutFn_I$ we see that that
$\widetilde{\sigma}$ conjugates $\AutFn_{\sigma(I)}$ to $\AutFn_{I}$
and hence conjugates $\SAutFn_{\sigma(I)}$ to $\SAutFn_{I}$. Since the index-2 subgroup $S_n^\pm\cap \SAutFn$
of $S_n^\pm$ surjects onto $S_n$, it follows that $\SAutFn$ is a transitive $[n]$-group.
\end{proof}

\para{Generating $\mathbf{G(1)}$}
The second condition in Corollary~\ref{corollary:heart} is that $G(1) = \IA_n$ is finitely generated.  This
was proved by Magnus.  For later use, we will actually give an explicit generating set.  Let
$\{x_1,\ldots,x_n\}$ be the standard basis for $F_n$.
For distinct $1 \leq i,j \leq n$, define $C_{ij} \in \IA_n$ via the formula
\[C_{ij}(x_{\ell}) = \begin{cases}
x_j^{-1} x_{\ell} x_j & \text{if $\ell=i$},\\
x_{\ell} & \text{if $\ell \neq i$}.\end{cases}\]
Also, for distinct $1 \leq i,j,k \leq n$ define $M_{ijk} \in \IA_n$ via the formula
\[M_{ijk}(x_{\ell}) = \begin{cases}
x_{\ell} [x_j, x_k] & \text{if $\ell = i$},\\
x_{\ell} & \text{if $\ell \neq i$}.\end{cases}\]
Magnus (\cite{MagnusGenerators}; see
\cite{BestvinaBuxMargalit} and \cite{DayPutmanCurve} for modern accounts) proved the following.

\begin{theorem}
\label{theorem:magnus}
For $n \geq 2$, the group $\IA_n$ is $\SAutFn$-normally generated by $C_{12}$ and
is generated by the finite set of all $C_{ij}$ and $M_{ijk}$.
\end{theorem}

\para{Zariski-irreducibility}
The third condition in Corollary~\ref{corollary:heart} is that the action of $\Gamma = \SAutFn$ on
each $\cL(k)$ is Zariski-irreducible, which is the content of the following.

\begin{lemma}
\label{lemma:sautfnzariski}
For all $k \geq 1$, the action of $\Gamma$ on $\cL(k)$ is Zariski-irreducible.
\end{lemma}

\noindent
For the proof of Lemma~\ref{lemma:sautfnzariski}, we need the following classical computation of $\IA_n^{\ab}$. Let $V_\Z=F_n^{\ab}\iso\Z^n$, and recall that the natural action of $\AutFn$ on $V_\Z$ factors through $\GL_n(\Z)$. The following description holds for all $n\geq 0$.

\begin{theorem}
\label{theorem:IAn-ab}
There is an $\AutFn$-equivariant isomorphism $\IA_n^{\ab} \iso \Hom(V_{\Z},\bwedge^2 V_{\Z})$.
\end{theorem}

As an abelian group, the description of $\IA_n^{\ab}$ in Theorem~\ref{theorem:IAn-ab} was  established by Bachmuth~\cite{Bachmuth} in 1966 and implicitly (and independently)
by Andreadakis~\cite{Andreadakis} in 1965. The description as an $\AutFn$-module must have been folklore for some time; the earliest proof in the literature that we are aware of is due to Formanek~\cite{Formanek}.

\begin{proof}[Proof of Lemma~\ref{lemma:sautfnzariski}]
Since $G(k)=\gamma_k G$ is the lower central series of $G$, by Remark~\ref{remark:Lgen1} it 
suffices to prove the lemma for $\cL(1)=\IA_n^{\ab} \otimes \R$.  By Theorem~\ref{theorem:IAn-ab}, the
action of $\SAutFn$ on $\IA_n^{\ab} \otimes \R$ factors through the surjection $\SAutFn\onto \SL_n(\Z)$. 
Since $\SL_n(\Z)$ is Zariski dense in $\SL_n(\R)$, arguing 
as in the proof of Lemma~\ref{lemma:torellizariski} in \S\ref{section:torellilcs}, we deduce that the image of 
$\SAutFn$ in $\GL(\IA_n^{\ab} \otimes \R)$ is irreducible.
\end{proof}

\para{Connectivity bounds}
The fourth condition in Corollary~\ref{corollary:heart} asserts that the graph $X_m(\Gamma)$ must be connected for some 
$m \geq \max\Set{$d_{\Gamma}(\cL(k))$}{$1 \leq k < N$}$.  This requires showing that $X_m(\Gamma)$ is
connected if $m$ is not too large (relative to $n$), and then estimating $d_{\Gamma}(\cL(k))$.  We
start with the first of these.

\begin{lemma}
\label{lemma:sautfnxm}
For all $m \geq 2$ such that $2m+1 \leq n$, the graph $X_m(\Gamma)$ is connected and nontrivial.
\end{lemma}

\noindent
For the proof of Lemma~\ref{lemma:sautfnxm}, we will need a generating set for $\SAutFn$.  
Let $\{x_1,\ldots,x_n\}$ be the standard basis for $F_n$.
For distinct $1 \leq i,j \leq n$, define $L_{ij} \in \SAutFn$ and $R_{ij} \in \SAutFn$ via the formulas
\[L_{ij}(x_{\ell}) = \begin{cases}
x_j x_{\ell} & \text{if $\ell = i$},\\
\phantom{x_j}x_{\ell} & \text{if $\ell \neq i$}.\end{cases}\qquad\qquad
R_{ij}(x_{\ell}) = \begin{cases}
x_{\ell}x_j  & \text{if $\ell = i$},\\
x_{\ell} & \text{if $\ell \neq i$}.\end{cases}
\]
Nielsen (\cite{NielsenGen}; see \cite{GerstenPres} for a modern account) proved the following.

\begin{theorem}
\label{theorem:nielsen}
For $n \geq 2$, the group $\SAutFn$ is generated by the set of all $L_{ij}$ and $R_{ij}$.
\end{theorem}

\begin{proof}[Proof of Lemma~\ref{lemma:sautfnxm}]
We will use Proposition~\ref{prop:connectivity}.  Let $S \subset \SAutFn$ be the set of all $L_{ij}$ and $R_{ij}$.  Setting
$c=\max\Set{$\comp(s)$}{$s \in S$}$, the fact that $\comp(L_{ij}) = \comp(R_{ij}) = 2$ implies that $c=2$.  
Lemma~\ref{lemma:sautfnbasic} says that $\SAutFn$ is transitive and commuting.  In particular, every element
of $S$ is good.  We can thus apply Proposition~\ref{prop:connectivity} and deduce that
$X_m(\Gamma)$ is connected whenever $2m+c-1 = 2m+1 \leq n$. 
As for the nontriviality of $X_m(\Gamma)$, it follows immediately from the fact
that $\SAut(F_m)$ is never normal in $\SAut(F_n)$ except with $m=0,1$ (when $\SAut(F_m) = 1$) and
when $m=n$ (when $\SAut(F_m) = \SAut(F_n)$); see Remark \ref{remark:nontrivial}.
\end{proof}

We now estimate $d_\Gamma(\cL(k))$.

\begin{lemma}
\label{lemma:sautfndgamma}
For all $k\geq 1$ we have $d_\Gamma(\cL(k))\leq 2k$.
\end{lemma}
\begin{proof}
We will use Proposition~\ref{prop:generalbound}.  This proposition requires that $\cL$ is generated
by $\cL(1)$, which holds since the filtration $G(k)$ is the lower central series.  We now calculate
the quantities that go into its bound:
\begin{compactitem}
\item Theorem~\ref{theorem:magnus} says that $\IA_n$ is $\SAutFn$-normally
generated by the element $C_{12}$.  Since $\comp(C_{12}) = 2$, we have $d_{\Gamma}(\cL(1)) \leq 2$. We cannot have $d_{\Gamma}(\cL(1))\leq 1$ (since $G_{\{i\}}=1$ and thus $\cL(1)_{\{i\}}=0$) so in fact $d_{\Gamma}(\cL(1)) = 2$.
\item Set $d = d(\cL(1))$.  Theorem~\ref{theorem:magnus} says that $\IA_n$ is generated
by the set of all $C_{ij}$ and $M_{ijk}$.  Since $\comp(C_{ij}) = 2$ and $\comp(M_{ijk}) = 3$, we deduce
that $d \leq 3$ (and one can check that in fact $d=3$).
\item Lemma~\ref{lemma:sautfnbasic} says that $\SAutFn$ is commuting, so as in Proposition~\ref{prop:generalbound}
we set $e = d-1 = 2$.
\end{compactitem}
Proposition~\ref{prop:generalbound} now says that for $k \geq 1$ we have 
\[d_\Gamma(\cL(k)) \leq d_\Gamma(\cL(1)) +(k-1)e = 2 + (k-1)2 = 2k.\qedhere\]
\end{proof}

\para{Putting it all together}
All the pieces are now in place to prove Theorem~\ref{maintheorem:ialcs}.  This theorem
has two parts that we prove separately.

\begin{proof}[Proof of Theorem~\ref{maintheorem:ialcs} for $k \geq 3$]
The notation is as above.  As was established in \S \ref{section:outlineproof}, we must
prove that $G(k)$ is finitely generated for $k \geq 3$ and $n \geq 4k-3$,
or equivalently when $k \leq \frac{n+3}{4}$.
We will apply Corollary~\ref{corollary:heart} with $N = \lfloor\frac{n+3}{4}\rfloor$ and $m=2(N-1)=2\cdot\lfloor\frac{n-1}{4}\rfloor$.  This theorem has four hypotheses: 
\begin{compactitem}
\item The $[n]$-group $\Gamma$ must be transitive, which is one of the conclusions of Lemma~\ref{lemma:sautfnbasic}.
\item The group $G$ must be finitely generated, which is Theorem~\ref{theorem:magnus}.
\item The action of $\Gamma$ on each $\cL(k)$ must be Zariski-irreducible, which is
Lemma~\ref{lemma:sautfnzariski}.
\item The graph $X_m(\Gamma)$ must be connected and nontrivial, and we must have
\begin{equation}
\label{eqn:autmbound}
m \geq \max\Set{$d_{\Gamma}(\cL(k))$}{$1 \leq k < N$}.
\end{equation}
To see that $X_m(\Gamma)$ is connected and nontrivial, it is enough to verify the two hypotheses
of Lemma~\ref{lemma:sautfnxm}.  The first is that $m \geq 2$; indeed, since $n\geq 4k-3\geq 9$, we have
\[m=2\cdot\lfloor\frac{n-1}{4}\rfloor\geq 4 \geq 2.\]
The second is that $2m+1 \leq n$; indeed,
\[2m+1=4\cdot \lfloor\frac{n-1}{4}\rfloor +1\leq 4\cdot \frac{n-1}{4} +1 = n.\]
As for \eqref{eqn:autmbound}, Lemma~\ref{lemma:sautfndgamma} says that
$d_{\Gamma}(\cL(k)) \leq 2k$, so
\[\max\Set{$d_{\Gamma}(\cL(k))$}{$1 \leq k < N$}\leq 2(N-1)=m,\] 
as desired.
\end{compactitem}
Applying Corollary~\ref{corollary:heart}, we conclude that $G(k)$ is finitely generated for $1 \leq k \leq N$.
\end{proof}

\begin{proof}[Proof of Theorem~\ref{maintheorem:ialcs} for $k = 2$]
We must prove that $[\IA_n,\IA_n]$ is finitely generated for $n \geq 4$.
To do this, we will apply  Theorem~\ref{thm:Kg-argument-general} to $G=\IA_n$ and $\Gamma=\SAutFn$ acting by conjugation. Our $\Gamma$-orbit $C$ will be the $\Gamma$-conjugates of the Magnus generator $C_{12}$; Theorem~\ref{theorem:magnus} tells us that $G$ is generated by $C$ and finitely generated. We have already checked in Lemma~\ref{lemma:sautfnzariski} that the action of $\Gamma$ on $\cL(1)\iso \Hom(G;\R)$ is Zariski-irreducible. Therefore we need only verify the remaining hypothesis of Theorem~\ref{thm:Kg-argument-general}: denoting by $\Conj(\IA_n)$ the graph whose vertices are $\SAutFn$-conjugates of $C_{12}$ with edges connecting commuting elements, we must show that $\Conj(\IA_n)$ is connected when $n\geq 4$.

Let $S=\{L_{ij}^\pm,R_{ij}^\pm\}$ be the generating set for $\SAutFn$ from Theorem~\ref{theorem:nielsen}. Just as in the proofs of Proposition~\ref{prop:BPgraph} or Proposition~\ref{prop:connectivity}, to prove that $\Conj(\IA_n)$ is connected, it suffices to prove the following: for all $s\in S$, there exists a path $\eta_s$ in $\Conj(\IA_n)$ from $C_{12}$ to ${C_{12}}^s$. We will repeatedly rely on the observation that for fixed $b\in [n]$, the $2n-2$ elements $\{L_{ib},R_{ib}\,|\,i\neq b\}$ commute; indeed, they generate a subgroup of $\SAutFn$ isomorphic to $\Z^{2n-2}$, which contains $C_{ib}=R_{ib}L_{ib}^{-1}$ for all $i\neq b$.

Fix distinct $a$ and $b$ in $[n]$ and consider $s\in \{L_{ab},L_{ab}^{-1},R_{ab},R_{ab}^{-1}\}$.
\begin{compactitem}
\item If $\{a,b\}\cap\{1,2\}=\emptyset$, then $s$ commutes with $C_{12}$; thus $(C_{12})^s=C_{12}$ and there is nothing to prove.
\item If $\{a,b\}=\{1,2\}$, then $C_{34}$ commutes with both $C_{12}$ and $s$, and thus with  $(C_{12})^s$. Therefore for $\eta_s$ we may take the length 2 path from $C_{12}$ to $C_{34}$ to $(C_{12})^s$. 
\item If $a\in \{1,2\}$ and $b\notin \{1,2\}$, we can choose some $c\in [n]\setminus\{1,2,b\}$ since $n\geq 4$. Then $C_{cb}$ commutes with  $C_{12}$ (since $\{1,2\}\cap \{c,b\}=\emptyset$). At the same time, $C_{cb}$ commutes with $s$ (since both lie in the abelian subgroup generated by $L_{ib}$ and $R_{ib}$), and thus with  $(C_{12})^s$. Therefore for $\eta_s$ we may take the length 2 path from $C_{12}$ to $C_{cb}$ to $(C_{12})^s$.
\item It remains to handle the case $b\in \{1,2\}$ and $a\notin \{1,2\}$.
\begin{compactitem}
\item If $b=2$, then $s$ commutes with $C_{12}$, so $(C_{12})^s=C_{12}$ and there is nothing to prove.
\item Finally, if $b=1$ and $a\notin \{1,2\}$, we can choose some $d\in [n]\setminus\{1,2,a\}$ since $n\geq 4$. Then $C_{d2}$ commutes with  $C_{12}$ (since both lie in the abelian subgroup generated by $L_{i2}$ and $R_{i2}$). At the same time, $C_{d2}$ commutes with $s$ (since $\{a,b\}\cap \{d,2\}=\emptyset$), and thus with $(C_{12})^s$. Therefore for $\eta_s$ we may take the length 2 path from $C_{12}$ to $C_{d2}$ to $(C_{12})^s$.
\end{compactitem}
\end{compactitem}
This concludes the proof that $\Conj(\IA_n)$ is connected for $n\geq 4$.  Theorem~\ref{thm:Kg-argument-general} now shows that $[\IA_n,\IA_n]=\gamma_2(\IA_n)$ is finitely generated for $n\geq 4$.
\end{proof}

\subsection{The Johnson filtration of \texorpdfstring{$\AutFn$}{AutFn}}
\label{section:iajohnson}

We close by describing how to modify the proof of Theorem~\ref{maintheorem:ialcs} from \S\ref{section:ialcs} to
prove Theorem~\ref{maintheorem:iajohnson}, which as we discussed in \S \ref{section:outlineproof} is equivalent
to the assertion that for $k \geq 2$ the term $\JIA_n(k)$ of the Johnson filtration of $\IA_n$ is finitely generated for $n \geq 2k+3$.  This is stronger than the bound from Theorem~\ref{maintheorem:ialcs}, which
only gives this for $n \geq 4k-3$.  

\para{Notation}
The following notation will be in place for the remainder of this subsection.  Fix some $n \geq 2$.
Let $\Gamma = \SAutFn$, let $G = \IA_n$, and let $G(k) = \JIA_n(k)$.  Let
$\cLJ = \bigoplus \cLJ(k)$ be the graded real Lie algebra associated to $G(k)$.  Maintaining the same notation from the previous subsection, let
$\cL = \bigoplus \cL(k)$ be the graded real Lie algebra associated to the lower central series of $\IA_n$.  
The groups $\Gamma$ and
$G$ and $G(k)$ are endowed with the $[n]$-group structure coming from the $[n]$-group structure on $\AutFn$, and
the vector spaces $\cLJ(k)$ and $\cL(k)$ are endowed with the induced $[n]$-vector space structures.

\para{What must be done}
The structure of the proof of Theorem~\ref{maintheorem:iajohnson} is exactly the same as that of Theorem~\ref{maintheorem:ialcs};
the only change needed is to use the following two results in place of Lemmas~\ref{lemma:sautfnzariski} and~\ref{lemma:sautfndgamma}, respectively.
\begin{lemma}
\label{lemma:sautfnjohnsonzariski}
For all $k\geq 1$ the action of $\Gamma$ on $\cLJ(k)$ is Zariski-irreducible.
\end{lemma}
\begin{proposition}
\label{prop:Bartholdibound} 
For all $k\geq 1$ we have $d_\Gamma(\cLJ(k))\leq k+2$.
\end{proposition}
The improved bound in Proposition~\ref{prop:Bartholdibound} (compared with the bound $d_\Gamma(\cL(k))\leq 2k$ in Lemma~\ref{lemma:sautfndgamma})  is the source of our improved range of finite generation for $\JIA_n(k)$. We will prove these two results below, but first we illustrate how they imply Theorem~\ref{maintheorem:iajohnson}.

\begin{proof}[Proof of Theorem~\ref{maintheorem:iajohnson}]
We must
prove that $G(k)=\JIA_n(k)$ is finitely generated for $n\geq 2k+3$, or equivalently when $k\leq \frac{n-3}{2}$. 
With notation as above, we will apply Corollary~\ref{corollary:heart} with $N = \lfloor\frac{n-3}{2}\rfloor$ and $m=N+1=\lfloor\frac{n-1}{2}\rfloor$.
The first two hypotheses, dealing with $\Gamma$ and $G(1) = \IA_n$, are unchanged from before, so we must verify the remaining two hypotheses.
\begin{compactitem}
\item The action of $\Gamma$ on each $\cLJ(k)$ must be Zariski-irreducible, which is
Lemma~\ref{lemma:sautfnjohnsonzariski}.
\item The graph $X_m(\Gamma)$ must be connected and nontrivial, and we must have
\begin{equation}
\label{eqn:autmbound2}
m\geq \max\Set{$d_{\Gamma}(\cLJ(k))$}{$1 \leq k < N$}.
\end{equation}
To see that $X_m(\Gamma)$ is connected and nontrivial, it is enough to verify the two hypotheses
of Lemma~\ref{lemma:sautfnxm}.  The first is that $m \geq 2$; indeed, since $n\geq 2k+3\geq 5$, we have
\[m=\lfloor\frac{n-1}{2}\rfloor\geq 2.\]
The second is that $2m+1 \leq n$; indeed,
\[2m+1=2\cdot \lfloor\frac{n-1}{2}\rfloor +1\leq 2\cdot \frac{n-1}{2} +1 = n.\]
As for \eqref{eqn:autmbound2}, Proposition~\ref{prop:Bartholdibound} says that
$d_{\Gamma}(\cLJ(k)) \leq k+2$, so
\[\max\Set{$d_{\Gamma}(\cL(k))$}{$1 \leq k < N$}\leq (N-1)+2=m,\] 
as desired.
\end{compactitem}
Applying Corollary~\ref{corollary:heart}, we conclude that $G(k)$ is finitely generated for $1 \leq k \leq N$.
\end{proof}

We now proceed to prove Lemma~\ref{lemma:sautfnjohnsonzariski} and Proposition~\ref{prop:Bartholdibound}.
The following key fact will be used in the proofs of both. 
Let $V_\Z=F_n^{\ab}\iso \Z^n$ and $V=V_\Z\otimes \R \iso \R^n$. Let $\Lie(V_\Z)=\bigoplus \Lie_m(V_\Z)$ be the free Lie algebra on $V_\Z$, and let $\Lie(V)=\bigoplus \Lie_m(V)$ be the free $\R$-Lie algebra on $V$, so that $\Lie_m(V)\iso \Lie_m(V_\Z)\otimes \R$. The action of $\Gamma$ on $V_\Z$ and $V$, which factors through $\SL(V_\Z)\iso \SL_{n}(\Z)$, extends to an action on $\Lie(V_\Z)$ and $\Lie(V)$ by Lie algebra automorphisms.  There is a canonical $\Gamma$-equivariant embedding \[\iota\colon \JIA_n[k]/\JIA_n[k+1]\into V_\Z^*\otimes \Lie_{k+1}(V_\Z),\] described concretely as follows. Given $\varphi\in \Aut(F_n)$, to say that $\varphi\in \cLJ(k)$ means by definition that the map $x\mapsto x^{-1}\varphi(x)$ defines a function $F_n\to \gamma_{k+1} F_n$. This descends to a homomorphism from $F_n^{\ab}=V_\Z$ to $\gamma_{k+1} F_n/\gamma_{k+2}F_n\iso \Lie_{k+1}(V_\Z)$. The resulting assignment $\JIA_n[k]\to V_\Z^*\otimes \Lie_{k+1}(V_\Z)$ is a homomorphism, and by definition its kernel is $\JIA_n[k+1]$.
Tensoring with $\R$, we obtain a $\Gamma$-equivariant embedding of $\cLJ(k)$ into $M(k)\coloneq V^*\otimes \Lie_{k+1}(V)$.
In particular, this equivariance implies that the action of $\Gamma$ on $\cLJ(k)$ factors through $\SL(V_\Z)$.

\begin{proof}[Proof of Lemma~\ref{lemma:sautfnjohnsonzariski}]
The action of $\SL(V_\Z)$ on $M(k)$ extends to a polynomial representation of $\SL(V)\iso \SL_n(\R)$. The $\Gamma$-equivariance of the embedding $\cLJ(k)\to M(k)$ implies that the subspace $\cLJ(k)$ is $\SL(V_\Z)$-invariant. Since $\SL(V_\Z)$ is Zariski-dense in $\SL(V)$ and the map $\GL(V)\to \GL(M(k))$ is Zariski-continuous, the subspace $\cLJ(k)$ must also be $\SL(V)$-invariant. Moreover, this implies that the Zariski closure of the image of $\Gamma$  in $\GL(\cLJ(k))$ coincides with the image of $\SL(V)$; it is therefore a quotient of $\SL(V)\iso \SL_n(\R)$, and thus is irreducible.
\end{proof}

\begin{proof}[Proof of Proposition~\ref{prop:Bartholdibound}]This proposition is a fairly easy consequence of Bartholdi's work \cite{Bar}, but in order to make our argument precise it will be more convenient to refer to other sources.

Note that when $n\leq k+2$, the proposition is vacuous; we may therefore assume that $k<n-2$. We will actually only assume $k\leq n-2$ since this suffices for the argument below. Throughout this section, we will write $[v_1,v_2,\ldots,v_k]$ for the \emph{left-normed} commutator:
\[[v_1,v_2,\ldots,v_k]\coloneq [[[v_1,v_2],\cdots],v_k].\]
For $k=1$, since $\JIA_n(2)=\gamma_2(\IA_n)$ we have already obtained the stronger bound 
\[d_\Gamma(\cLJ(1))=d_\Gamma(\cL(1))\leq 2\] 
in Lemma~\ref{lemma:sautfndgamma}. So fix $2\leq k\leq n-2$, and consider the following finite families of automorphisms:
\begin{compactitem}
\item[(1)] Let $i\in [n]$ and $\omega=\omega_1\omega_2\ldots\omega_{k+1}$ be a sequence of length $k+1$ with $i\notin \omega$ (that is,
$\omega_j\in [n]\setminus\{i\}$ for each $j$). Let $T_{i,\omega}$ be the element of $\AutFn$ which sends
$x_i$ to $x_i[x_{\omega_1},x_{\omega_2},\ldots, x_{\omega_{k+1}}]$ 
and fixes $x_j$ for all $j\neq i$.
\item[(2)] Let $\mu=\mu_1\mu_2\ldots\mu_{k}$ be a sequence of length $k$ with $\mu_j\in [n]$ for each $j$.
For each $\mu$, choose once and for all two distinct elements $i, j\in [n]$ with $i\notin \mu$ and $j\notin \mu$ (this is possible precisely because $k\leq n-2$).
Define $S_{\mu}$ to be the left-normed commutator 
\[S_{\mu}=[M_{ij\mu_1}, C_{i\mu_2},C_{i\mu_3},\ldots,C_{i\mu_{k-1}},M_{ji\mu_k}],\]
where  $C_{ij}$ and $M_{ijl}$ are the Magnus generators defined before Theorem~\ref{theorem:magnus}. Note that the definition of $S_{\mu}$ depends on the choice of $i$ and $j$, but this dependence will not be important for our purposes.
\end{compactitem} 

By construction, the elements $T_{i,\omega}$ and $S_{\mu}$ have complexity at most $k+2$ and lie in $G(k)$. Let $t_{i,\omega}$ and $s_\mu$ be their images in $\cLJ(k)$.
Denote by $A$ (resp.\ $B$) the subspace of $\cLJ(k)$ generated by the $\Gamma$-orbits of the elements $t_{i,\omega}$ (resp.\ by the $\Gamma$-orbits of the elements $s_{\mu}$). Since $t_{i,\omega}$ and $s_{\mu}$ have complexity at most $k+2$, to prove that $d_\Gamma(\cLJ(k))\leq k+2$ it suffices to show that $\cLJ(k)=A+B$. 

Define the map $\Phi\colon V^*\otimes V^{\otimes k+1}\to V^{\otimes k}$ by 
\[\Phi(v^*\otimes v_0\otimes\cdots \otimes v_{k})=v^*(v_0)v_1\otimes\cdots \otimes v_{k}.\]
Realizing $\Lie_{k+1}(V)$ as a subspace of $V^{\otimes k+1}$ in the standard way, we obtain a composite $\Gamma$-equivariant map
\[\tau\colon \cLJ(k)\into M(k)=V^*\otimes \Lie_{k+1}(V)\into V^*\otimes V^{\otimes k+1}\stackrel{\Phi}{\longrightarrow} V^{\otimes k}.\]
Let $\langle \gamma \rangle$ be a cyclic group of order $k$ acting on $V^{\otimes k}$ by cyclically permuting the factors, that is,
$\gamma(v_1\otimes\cdots \otimes v_{k-1}\otimes v_{k})=v_2\otimes\cdots \otimes v_{k}\otimes v_1$. Let $W$ be the subspace of
$\gamma$-invariant elements in $V^{\otimes k}$.
We claim that
\begin{compactenum}[label=(\roman*)]
\item\label{part:Aker} $A=\Ker(\tau)$
\item\label{part:tauL} $\tau(\cLJ(k))\subseteq W$
\item\label{part:tauB} $\tau(B)\supseteq W$
\end{compactenum}
Claims \ref{part:tauL} and \ref{part:tauB} together imply that  $\tau(B)=W=\tau(\cLJ(k))$, so with \ref{part:Aker} this implies that $\cLJ(k)=A+B$ as desired.

For $i\in [n]$ let $e_i\in V=F_n^{\ab}\otimes \R$ be the image of $x_i$, and let ${e_1^*,\ldots,e_n^*}$ be the dual basis of $V^*$. For any sequence $\delta=i\omega=i\omega_1\omega_2\ldots\omega_{k+1}$ of length $k+2$, define 
\[e_{\delta}=e_i^*\otimes e_{\omega}=e_i^*\otimes [e_{\omega_1},e_{\omega_2},\ldots,e_{\omega_{k+1}}].\] 
Note that these elements span $V^*\otimes \Lie_{k+1}(V)$.

The assertion~\ref{part:tauL} is merely a restatement of  \cite[Prop~5.3]{MassuyeauSakasai}. For~\ref{part:tauB}, an easy direct computation shows that
for any sequence $\mu=\mu_1\cdots \mu_k$ of length $k$ we have 
\[
\tau(s_{\mu})=
e_{\mu_1}\otimes \cdots\otimes e_{\mu_{k-1}}\otimes e_{\mu_{k}}-e_{\mu_2}\otimes \cdots\otimes e_{\mu_k}\otimes e_{\mu_{1}} 
\]
Since elements of this form span $W$, inclusion~\ref{part:tauB} follows.

It remains to prove~\ref{part:Aker}. 
Since $\Ker(\tau)$ is $\Gamma$-invariant, to verify the inclusion $A\subseteq \Ker(\tau)$ we only need to check that $\tau(t_{i,\omega})=0$. But since $i\notin \omega$, as an element of $V^*\otimes \Lie_{k+1}(V)$ we have \[t_{i,\omega}=e_{i\omega}=e_i^*\otimes [e_{\omega_1},e_{\omega_2},\ldots,e_{\omega_{k+1}}].\] Since $\omega$ does not contain $i$, the term $[e_{\omega_1},e_{\omega_2},\ldots,e_{\omega_{k+1}}]$ belongs to $\Lie_{k+1}(\ker e_i^*)$, so $t_{i,\omega}$ belongs to $\Ker(\Phi)$. This shows that $\tau(t_{i,\omega})=0$ as claimed, so $A\subseteq \Ker(\tau)$.
The opposite inclusion $\Ker(\tau)\subseteq A$ (which is what we ultimately need) is implicitly proved in 
\cite[Lemma~5.4]{Bar} and also in \cite[Prop~3.2]{Satoh}, but for clarity we will give a short direct proof.

Given any sequence $\delta=i\omega=i\omega_1\omega_2\ldots\omega_{k+1}$ of length $k+2$, let $c({\delta})$ denote the number of times
$i$ (the first element of $\delta$) appears in the tail $\omega_1\ldots \omega_{k+1}$.
We already observed that $e_{\delta}\in A$ when $c({\delta})=0$ since then $e_{\delta}=t_{i,\omega}$.

We next claim that if $c({\delta})\geq 2$, then we can write 
$e_{\delta}=x+y$ where $x\in A$ and $y$ is a linear combination of elements $e_{\gamma}$ with $c({\gamma})<c({\delta})$.
Indeed, choose any index $j\in [n]$ with $j\notin \delta$, which is possible since the number of distinct elements
of $\delta$ is at most $k+2-c(\delta)\leq n-c(\delta)\leq n-2$. Let $\omega'$ be the sequence obtained from $\omega$ by replacing all appearances
of $i$ by $j$ and let $\delta'=i\omega'$. Note that $c(\delta')=0$ so $e_{\delta'}\in A$.

Let $E_{ij}\in\Gamma$ be any element projecting to the elementary matrix $E_{ij}\in \SL_n(\Z)$.
For simplicity, we will write the action of $\Gamma$ on the left. Set 
$z=E_{ij}^2 e_{\delta'}-2E_{ij}e_{\delta'}$, so $z\in A$. The action of $E_{ij}^p$ on 
$e_{\delta'}$ will replace $e_i^*$ by $e_i^*-pe_j^*$ and each occurrence of $e_j$ by $e_j+pe_i$, and a simple computation
shows that $z=(2^{c({\delta})}-2)e_{\delta}+u$ where $c({\gamma})<c({\delta})$ for each $e_{\gamma}$ that appears in $u$.
Since $c({\delta})>1$, we have $e_{\delta}=\frac{1}{2^{c({\delta})}-2}(z-u)$, as desired.

Applying this claim inductively shows that any element of $V^*\otimes \Lie_{k+1}(V)$ can be written $z=a+b$, where
$a\in A$ and $b=\sum \lambda_{\delta}e_{\delta}$ with $c({\delta})=1$ for each $\delta$. Moreover, using the Lie algebra
axioms, we can assume that each $\delta$ in the above sum has the form $\delta=ii\varepsilon$ where $\varepsilon$
is a sequence of length $k$ with $i\notin \varepsilon$.

Note that  $\Phi(e_{ii\varepsilon})=\Phi(e_i^*\otimes e_{i\varepsilon})$ is equal to $e_{\varepsilon_1}\otimes e_{\varepsilon_2}\otimes \cdots\otimes e_{\varepsilon_{k}}$ when $i\notin \varepsilon$. To see this, note that when the left-normed commutator $e_\omega=[[[e_{\omega_1},e_{\omega_2}],\cdots],e_{\omega_\ell}]\in \Lie_{k+1}(V)$ is considered as an element of $V^{\otimes k+1}$, it is equal to $e_{\omega_1}\otimes e_{\omega_2}\otimes \cdots\otimes e_{\omega_{k+1}}$ plus permutations of the form $e_{\omega_\ell}\otimes \cdots$ for $\ell\neq 1$. 

Now take an arbitrary $z\in \Ker(\tau)$ and write it as a sum $z=a+b$ as above. Since $A\subseteq \Ker(\tau)$, this implies that $\tau(b)=0$. 
Since $\Phi(e_{ii\varepsilon}) = e_{\varepsilon_1} \otimes e_{\varepsilon_2}\otimes \cdots\otimes e_{\varepsilon_k}$ for $i\notin \varepsilon$ and all such simple tensors coming from different $\varepsilon$ are linearly independent, 
to have $\tau(b)=0$  means that for each $\varepsilon$ we have $\sum_{i\not\in \varepsilon}\lambda_{e_{ii\varepsilon}}=0$. It follows that $b$ must be a linear combination
of elements of the form $e_{ii\varepsilon}-e_{jj\varepsilon}$ with $i\neq j$ and $i,j\not\in \varepsilon$. 
However, these elements too belong to $A$:
\[e_{ii\varepsilon}-e_{jj\varepsilon}=E_{ij}e_{ij\varepsilon}-e_{ij\varepsilon}+e_{ji\varepsilon}=E_{ij}t_{i,j\varepsilon}-t_{i,j\varepsilon}+t_{j,i\varepsilon}\in A\]
We conclude that $b\in A$ and hence $z\in A$, as desired.
\end{proof}

\noindent
\begin{tabularx}{\linewidth}[t]{XXX}
{\raggedright
Thomas Church\\
Department of Mathematics\\
Stanford University\\
450 Serra Mall\\
Stanford, CA 94305\\
\myemail{tfchurch@stanford.edu}}
&
{\raggedright
Mikhail Ershov\\
Department of Mathematics\\
University of Virginia\\
141 Cabell Drive\\
Charlottesville, VA 22904\\
\myemail{ershov@virginia.edu}}
&
{\raggedright
Andrew Putman\\
Department of Mathematics\\
University of Notre Dame\\
279 Hurley Hall\\
Notre Dame, IN 46556\\
\myemail{andyp@nd.edu}}
\end{tabularx}

\begin{thebibliography}{KoMcCaMe}
\begin{footnotesize}
\setlength{\itemsep}{1pt}

\bibitem[A]{Andreadakis}
S. Andreadakis, On the automorphisms of free groups and free nilpotent groups, \emph{Proc. London Math. Soc.} (3) 15 (1965), 239--268.

\bibitem[Bac]{Bachmuth} S. Bachmuth, Induced automorphisms of free groups and free metabelian groups, \emph{Trans. Amer. Math. Soc.}  122  (1966), 1--17.

\bibitem[Bar1]{Bar}
L. Bartholdi, Automorphisms of free groups. I. \emph{New York J. Math.} 19 (2013), 395--421.

\bibitem[Bar2]{BartholdiErratum}
L. Bartholdi, Automorphisms of free groups. I -- erratum. \emph{New York J. Math.} 22 (2016), 1135--1137. \arXiv{1304.0498}.

\bibitem[BBM]{BestvinaBuxMargalit} 
M. Bestvina, K. U. Bux, and D. Margalit, 
Dimension of the Torelli group for $\Out(F_n)$,
\emph{Invent. Math.} 170 (2007), no. 1, 1--32.


\bibitem[BieNeSt]{BieriNeumannStrebel}
R. Bieri, W. D. Neumann, and R. Strebel, A geometric invariant of discrete groups, \emph{Invent. Math.}  90  (1987),  no. 3, 451--477.


\bibitem[CaSeMac]{CarterSegalMacdonald}
R. Carter, G. Segal, and I. Macdonald, {\it Lectures on Lie groups and Lie algebras}, London Mathematical Society Student Texts, 32, Cambridge University Press, Cambridge, 1995.

\bibitem[CP]{CP}
T. Church\ and\ A. Putman, Generating the Johnson filtration, \emph{Geom. Topol.} 19 (2015), no.~4, 2217--2255. \arXiv{1311:7150}.


\bibitem[DaP1]{DayPutmanCurve}
M. Day\ and\ A. Putman, The complex of partial bases for $F_n$ and finite generation of the Torelli subgroup of $\Aut(F_n)$, \emph{Geom. Dedicata} 164 (2013), 139--153. \arXiv{1012.1914}.

\bibitem[DaP2]{DayPutmanH2IA}
M. Day\ and\ A. Putman, On the second homology group of the Torelli subgroup of $\Aut(F_n)$, \emph{Geom. Topol.} 
21 (2017), no. 5, 2851--2896. \arXiv{1408.6242}.

\bibitem[De]{DehnGen}
M. Dehn, Die Gruppe der Abbildungsklassen, \emph{Acta Math.} 69 (1938), no.~1, 135--206.

\bibitem[DiPa]{DimcaPapadimaKg}
A. Dimca\ and\ S. Papadima, Arithmetic group symmetry and finiteness properties of Torelli groups, \emph{Ann. of Math.} (2) 177 (2013), no.~2, 395--423. \arXiv{1002.0673}.

\bibitem[DiHaPa]{DimcaPapadimaHainKg}
A. Dimca, R. Hain, and S. Papadima, The abelianization of the Johnson kernel, \emph{J. Eur. Math. Soc. (JEMS)} 16 (2014), no.~4, 805--822. \arXiv{1101.1392}.

\bibitem[EH]{EH}
M. Ershov and S. He, On finiteness properties of the Johnson filtrations, 
\emph{Duke Math. J.} 167 (2018), no. 9, 1713--1759. \arXiv{1703.04190}.




\bibitem[FMar]{FarbMargalitPrimer}
B. Farb\ and\ D. Margalit, {\it A primer on mapping class groups}, Princeton Mathematical Series, 49, Princeton Univ. Press, Princeton, NJ, 2012.

\bibitem[Fo]{Formanek} 
E. Formanek, Characterizing a free group in its automorphism group, \emph{J. Algebra} 133 (1990), no. 2, 424--432.

\bibitem[GaL]{GaroufalidisLevine}
S. Garoufalidis\ and\ J. Levine, Finite type $3$-manifold invariants and the structure of the Torelli group. I, \emph{Invent. Math.} 131 (1998), no.~3, 541--594.

\bibitem[Ge]{GerstenPres}
S. M. Gersten, A presentation for the special automorphism group of a free group, \emph{J. Pure Appl. Algebra} 33 (1984), no.~3, 269--279.

\bibitem[Ha1]{HainCompletions}
R. Hain, Completions of mapping class groups and the cycle $C-C^-$, in \emph{Mapping class groups and moduli spaces of Riemann surfaces (G\"{o}ttingen, 1991/Seattle, WA, 1991)}, 
 75--105, Contemp. Math., 150, Amer. Math. Soc., Providence, RI,  1993. 

\bibitem[Ha2]{HainInfinitesimal}
R. Hain, Infinitesimal presentations of the Torelli groups, \emph{J. Amer. Math. Soc.} 10 (1997), no.~3, 597--651. 

\bibitem[HatV]{HatcherVogtmannTethers}
A. Hatcher\ and\ K. Vogtmann, Tethers and homology stability for surfaces, \emph{Algebr. Geom. Topol.} 17 (2017), no.~3, 1871--1916. \arXiv{1508.04334}.

\bibitem[J1]{JohnsonHomeo}
D. Johnson, Homeomorphisms of a surface which act trivially on homology, \emph{Proc. Amer. Math. Soc.} 75 (1979), no.~1, 119--125. 

\bibitem[J2]{JohnsonSurvey}
D. Johnson, A survey of the Torelli group, in {\it Low-dimensional topology (San Francisco, Calif., 1981)}, 165--179, Contemp. Math., 20, Amer. Math. Soc., Providence, RI, 1983.

\bibitem[J3]{JohnsonFinite}
D. Johnson, The structure of the Torelli group. I. A finite set of generators for ${\mathcal I}$, \emph{Ann. of Math.} (2) 118 (1983), no.~3, 423--442.

\bibitem[J4]{JohnsonKer}
D. Johnson, The structure of the Torelli group. II. A characterization of the group generated by twists on bounding curves, \emph{Topology} 24 (1985), no.~2, 113--126.

\bibitem[J5]{JohnsonAbel}
D. Johnson, The structure of the Torelli group. III. The abelianization of $\mathcal{T}$, \emph{Topology} 24 (1985), no.~2, 127--144.


\bibitem[KoMcCaMe]{KMM-BNS}
N. Koban, J. McCammond, and J. Meier, The BNS-invariant for the pure braid groups, \emph{Groups Geom. Dyn.} 9 (2015), no.~3, 665--682. \arXiv{1306.4046}.

\bibitem[McCuMi]{McCulloughMiller}
D. McCullough\ and\ A. Miller, The genus $2$ Torelli group is not finitely generated, \emph{Topology Appl.} 22 (1986), no.~1, 43--49.

\bibitem[McCo]{McCool}
J. McCool, Generating the mapping class group (an algebraic approach), 
\emph{Publ. Mat. 40} (1996), no. 2, 457--468. 

\bibitem[Mag]{MagnusGenerators}
W. Magnus, \"Uber $n$-dimensionale Gittertransformationen, \emph{Acta Math.} 64 (1935), no.~1, 353--367.

\bibitem[MasSak]{MassuyeauSakasai}
G. Massuyeau\ and\ T. Sakasai, Morita's trace maps on the group of homology cobordisms, \emph{J. Topol. Anal.} 12 (2020), no.~3, 775--818. \arXiv{1606.08244}.

\bibitem[Mat]{Matsumoto}
M. Matsumoto, Introduction to arithmetic mapping class groups, in {\it Moduli spaces of Riemann surfaces}, 319--356, IAS/Park City Math. Ser., 20, Amer. Math. Soc., Providence, RI, 2013.

\bibitem[MeVW]{MeierVanWyk}
J. Meier\ and\ L. VanWyk, The Bieri-Neumann-Strebel invariants for graph groups, \emph{Proc. London Math. Soc.} (3) 71 (1995), no.~2, 263--280.

\bibitem[Mo1]{MoritaProblems} S. Morita, Problems on the structure of the mapping class group of surfaces and the topology of the moduli space of curves, in \emph{Topology, geometry and field theory}, 101--110, World Sci. Publ., River Edge, NJ,  1994. 

\bibitem[Mo2]{MoritaSurveyProspect} S. Morita, Structure of the mapping class groups of surfaces: a survey and a prospect, in
 \emph{Proceedings of the Kirbyfest (Berkeley, CA, 1998)}, 
 349--406, Geom. Topol. Monogr., 2, Geom. Topol. Publ., Coventry,  1999. 

\bibitem[MoSakSuz]{MoritaSakasaiSuzuki}
S. Morita, T. Sakasai\ and\ M. Suzuki, Torelli group, Johnson kernel, and invariants of homology spheres, \emph{Quantum Topol.} 11 (2020), no.~2, 379--410.

\bibitem[Ni]{NielsenGen}
J. Nielsen, Die Isomorphismengruppe der freien Gruppen, \emph{Math. Ann.} 91 (1924), no.~3-4, 169--209.

\bibitem[PaSu]{PapadimaSuciuKg}
S. Papadima\ and\ A. I. Suciu, Homological finiteness in the Johnson filtration of the automorphism group of a free group, \emph{J. Topol.} 5 (2012), no.~4, 909--944. \arXiv{1011.5292}.

\bibitem[Pe]{Pettet}
A. Pettet, The Johnson homomorphism and the second cohomology of ${\rm IA}_n$, \emph{Algebr. Geom. Topol.} 5 (2005), 725--740.  \arXiv{math/0501053}.

\bibitem[P1]{PutmanConnectivityNote}
A. Putman, A note on the connectivity of certain complexes associated to surfaces, \emph{Enseign. Math.} (2) 54 (2008), no.~3-4, 287--301. \arXiv{math/0612762}.

\bibitem[P2]{PutmanSmallGenset}
A. Putman, Small generating sets for the Torelli group, \emph{Geom. Topol.} 16 (2012), no.~1, 111--125. \arXiv{1106.3294}.

\bibitem[P3]{PutmanPartial}
A. Putman, Partial Torelli groups and homological stability, preprint 2019.\hfill\\ \arXiv{1901.06624}.

\bibitem[PSam]{PutmanSamLinear}
A. Putman\ and\ S. V. Sam, Representation stability and finite linear groups, \emph{Duke Math. J.} 166 (2017), no.~13, 2521--2598. \arXiv{1408.3694}.

\bibitem[Sat1]{Satoh}
T. Satoh, On the lower central series of the IA-automorphism group of a free group, 
\emph{J. Pure Appl. Algebra} 216 (2012), no. 3, 709--717.

\bibitem[Sat2]{SatohJIA}
T. Satoh, The third subgroup of the Andreadakis-Johnson filtration of the automorphism group of a free group, \emph{J. Group Theory} 22 (2019), no.~1, 41--61.

\bibitem[Ser]{SerreLie}
J.-P. Serre, {\it Lie algebras and Lie groups}, Lectures given at Harvard University, 1964, W. A. Benjamin, Inc., New York, 1965. 

\bibitem[Str]{StrebelBook}
R. Strebel, Notes on the Sigma invariants, preprint. \arXiv{1204.0214}.


\end{footnotesize}
\end{thebibliography}
\end{document}